\documentclass[reqno,11pt]{amsart}
\usepackage{amsmath,amsfonts,amssymb,amsthm,mathrsfs}
\usepackage{enumerate}
\allowdisplaybreaks[4]
\usepackage[T1]{fontenc}
\usepackage{times}
\usepackage[numbers, sort&compress]{natbib}
\usepackage{hyperref}
\usepackage[left=2.5cm,right=2.5cm,top=2.5cm,bottom=2.5cm]{geometry}
\hypersetup{
	colorlinks,
	citecolor=blue,
	filecolor=blue,
	linkcolor=blue,
	urlcolor=black
}

\numberwithin{equation}{section}
\newtheorem{thm}{Theorem}[section]
\newtheorem{lem}[thm]{Lemma}
\newtheorem{prop}[thm]{Proposition}
\newtheorem{defn}[thm]{Definition}
\newtheorem{exa}[thm]{Example}
\newtheorem{cor}[thm]{Corollary}
\newtheorem{rem}[thm]{Remark}
\theoremstyle{plain}

\renewcommand{\d}{\mathrm{d}}

\newcommand{\R}{\mathbb{R}}

\newcommand{\RN}{\mathbb{R}^N}

\newcommand{\lam}{\lambda}
\newcommand{\pa}{\partial}

\newcommand{\var}{\varepsilon}
\newcommand{\be}{\begin{equation}}
\newcommand{\ee}{\end{equation}}

\newcommand{\kap }{\kappa}
\newcommand{\si}{\sigma}

\newcommand{\om}{\omega}
\newcommand{\Om}{\Omega}
\newcommand{\diver}{\operatorname{div}}
\def\Xint#1{\mathchoice
    {\XXint\displaystyle\textstyle{#1}}%
    {\XXint\textstyle\scriptstyle{#1}}%
    {\XXint\scriptstyle\scriptscriptstyle{#1}}%
    {\XXint\scriptscriptstyle\scriptscriptstyle{#1}}%
    \!\int}
\def\XXint#1#2#3{{\setbox0=\hbox{$#1{#2#3}{\int}$}
        \vcenter{\hbox{$#2#3$}}\kern-.5\wd0}}

\def\dashint{\Xint-}
\newcommand{\cG}{\mathcal G}
\newcommand{\SobCap}{\operatorname{Cap}}
\newcommand{\capFN}{\operatorname{cap}_{F,N}}
\newcommand{\QN}{Q_N}
\newcommand{\W}{W}

\title[Sharp signed mass and rigidity]
{Sharp Signed Mass and Rigidity for Critical Finsler--Liouville Equations}

\author{Yanjun Liu}
\address{School of Mathematical Sciences, Chongqing Normal University, Chongqing 401331, China}
\email{liuyj@mail.nankai.edu.cn}

\author{Zhongwei Tang}
\address{School of Mathematical Sciences, Laboratory of Mathematics and
Complex Systems, Ministry of Education, Beijing Normal University,
Beijing 100875, China}
\email{tangzw@bnu.edu.cn}

\author{Heming Wang}
\address{State Key Laboratory of Mathematical Sciences, Academy of
Mathematics and Systems Science, Chinese Academy of Sciences,
Beijing 100190, China}
\email{hmwang@amss.ac.cn}

\date{}

\begin{document}

\begin{abstract}
We study bounded-above entire solutions of critical Finsler--Liouville equations with a possibly sign-changing prescribed function and finite absolute mass. Under a lower cone-growth condition, we establish a sharp signed-mass bound without upper growth or monotonicity assumptions. With an additional growth condition, annular compactness for a sequence of dilations yields Pohozaev boundary limits and rigidity for radially nonincreasing prescribed functions. Explicit examples show that the rigidity threshold is optimal. In the Euclidean plane, the mass estimate also gives a sharp obstruction to completeness.

\end{abstract}

\maketitle

 {\small\tableofcontents}

\section{Introduction}\label{sec:introduction}

The Liouville equation $-\Delta w=e^w$ in $\R^2$ is a classical model for the classification of
entire solutions.
Liouville~\cite{Liouville1853} obtained a local analytic representation.
Chen and Li~\cite{ChenLi1991} classified entire solutions satisfying $\int_{\R^2}e^w\,\d x<\infty$
by the method of moving planes, obtaining the standard radial family up to translation.
Chou and Wan~\cite{ChouWan1994} later used complex analysis, whereas
Chanillo and Kiessling~\cite{ChanilloKiessling1995} combined a Pohozaev
identity with the isoperimetric inequality.

Esposito~\cite{Esposito2018} classified solutions of
$-\Delta_Nw=e^w$ in $\RN$, $N\ge2$, satisfying $\int_{\RN}e^w\,\d x<\infty$,
where $\Delta_N w=\diver(|\nabla w|^{N-2}\nabla w)$ is the Euclidean
$N$-Laplacian.
The proof uses inversion to study infinity, relying on the conformal
invariance of the $N$-Dirichlet energy~\cite{Loewner1959}.
For a general anisotropy, Euclidean inversion does not preserve the operator.

Anisotropy is described by a convex, positively one-homogeneous function
$F:\RN\to[0,\infty)$, positive away from zero and of class
$C^2(\RN\setminus\{0\})$, that measures gradients in place of the
Euclidean norm. The associated Finsler $p$-Laplacian is
\begin{equation*}\tag{F}\label{F-Laplacian}
 Q_pw=\diver(F(\nabla w)^{p-1}F_\xi(\nabla w)),\quad p\in (1,\infty),
\end{equation*}
where $F_\xi$ denotes the gradient of $F$. The vector field
$F(\xi)^{p-1}F_\xi(\xi)$ is extended continuously by zero at $\xi=0$.
The operator is associated with the energy
$p^{-1}\int_D F(\nabla w)^p\,\d x$ on a bounded domain $D$.
For $p\in (1,N)$, Ciraolo, Figalli, and Roncoroni~\cite{CiraoloFigalliRoncoroni2020}
classified positive finite-energy solutions of
\[
 -Q_pw=w^{p^*-1}\quad\text{ in }\, \RN,\quad p^*=\frac{Np}{N-p},
\]
under smoothness and uniform convexity assumptions on $F$. Their theorem
also includes convex cones with the natural Neumann boundary condition.
Here finite energy means $\int_{\RN}(w^{p^*}+F(\nabla w)^p)\,\d x<\infty$.

Recently, Zhang~\cite{Zhang2026}
established a local Schoen-type Harnack inequality and classified
positive entire weak solutions $w\in W^{1,p}_{loc}(\RN)$ of
\[
-\Delta_p w=w^{p^*-1}h(w)\quad\text{ in }\,\RN,
\qquad p^*=\frac{Np}{N-p},\quad p\in (1,N),
\]
where $\Delta_p w=\diver(|\nabla w|^{p-2}\nabla w)$ and
$h\in C((0,\infty))$ is nonincreasing, with
$0<c_h\le h(s)\le C_h<\infty$ for all $s>0$.
The classification does not assume finite energy.
Here monotonicity concerns the solution variable, in contrast to
the spatial monotonicity of $V$ considered below.
For the anisotropic equation $-Q_pw=w^{p^*-1}$ in $\RN$,
Le~\cite{Le2026} obtained corresponding
classification and Schoen-type Harnack results under smoothness
and uniform convexity assumptions on $F$, again without a
finite-energy assumption.

The anisotropic $N$-Liouville equation is the critical-dimensional
counterpart of the Sobolev-power equation above. Ciraolo and
Li~\cite{CiraoloLi2024} classified solutions of
$-Q_Nw=e^w$ in $\RN$ with $\int_{\RN}e^w\,\d x<\infty$,
including nonreversible anisotropies, and resolved
the planar uniqueness conjecture of Wang and
Xia~\cite[Conjecture~6.1]{WangXia2012}. More recently, Ciraolo, Esposito,
and Li~\cite{CiraoloEspositoLi2026} developed a
$P$-function proof of the classification for
\[
 -\Delta_N w=|x|^{N\alpha}e^w\quad\text{ in }\,\RN,\quad
 \int_{\RN}|x|^{N\alpha}e^w\,\d x<\infty.
\]
Their method recovers the planar classification of Prajapat and
Tarantello~\cite{PrajapatTarantello2001} for $\alpha>-1$ and establishes the higher-dimensional classification for $N\ge3$ and
$-1<\alpha\le0$. For $\alpha<0$ the solutions are radial about the origin;
for $\alpha=0$ translations are also possible. The weight is positive
away from zero and may be singular at zero; this setting differs from
the $C^1$ sign-changing prescribed functions considered below.
Entire solutions also occur as rescaled profiles in blow-up problems for
the Finsler $N$-Laplacian; see, e.g.,
\cite{WangXia2012,XieGong2016,CiraoloLi2024} and the references therein.
Related geometric questions for the Finsler $p$- and $N$-Laplacians
include capacitary characterizations of Wulff shapes
\cite{BianchiniCiraolo2018,CiraoloLi2022} and anisotropic
Moser--Trudinger inequalities~\cite{ZhouZhou2019}.

In this paper, we consider the critical Finsler--Liouville equation
\be \label{maineq}
 -Q_Nu=V(x)e^{Nu}\quad\text{ in }\, \RN,\quad N\ge2,
\ee
where $V$ may change sign and $Q_N$ is given by \eqref{F-Laplacian} with $p=N$. We call $V$ the \emph{prescribed function}
throughout; the term \emph{prescribed curvature} is used when a geometric
curvature is identified with $V$, as in the planar case below.

Throughout the paper, we assume $N\ge2$ and that $F:\RN\to[0,\infty)$ satisfies
\begin{equation*}\tag{$*$}\label{ass:F}
\begin{aligned}
 &F\in C^2(\RN\setminus\{0\}),\quad F\text{ is convex},\quad
   F(\xi)>0\quad(\xi\ne0),\\
 &F(t\xi)=|t|F(\xi)\quad(t\in\R,\ \xi\in\RN),\\
 &\Big\langle D^2\Big(\frac{F^2}{2}\Big)(\xi)\eta,\eta\Big\rangle
   \ge\lam_F|\eta|^2\quad(\xi\ne0,\ \eta\in\RN)
\end{aligned}
\end{equation*}
for some $\lam_F>0$. Here $D$ and $D^2$ denote the gradient and Hessian,
and $\langle\cdot,\cdot\rangle$ is the Euclidean inner product.
The identity $F(-\xi)=F(\xi)$ is called \emph{reversibility}.

The dual norm, Wulff balls, and unit Wulff volume are
\[
 F^0(x)=\sup_{\xi\ne0}\frac{x\cdot\xi}{F(\xi)},\quad
 W_r(x_0)=\{x:F^0(x-x_0)<r\},\quad W_r=W_r(0),\quad
 W_F=W_1,\quad\kap _N=|W_1|.
\]
The dual norm also satisfies $F^0(-x)=F^0(x)$. We reserve $B_r(x_0)$ for Euclidean balls, set $B_r=B_r(0)$ and
$\mathbb S^{N-1}=\pa  B_1$, and write $|E|$ for Lebesgue measure.

Set $\Phi_N(\xi)=F(\xi)^N/N$ and $a_N=D\Phi_N$, with $a_N(0)=0$.
For $F(\xi)=|\xi|$, equation~\eqref{maineq} reduces to the Euclidean
$N$-Liouville equation. Its planar interpretation follows from the
conformal curvature law
$-\Delta_{g_0}\varphi+K_{g_0}=K_g e^{2\varphi}$ for $g=e^{2\varphi}g_0$
on a surface, where $K_g$ and $\d A_g$ denote Gaussian curvature and
area measure. For $N=2$ and $g_0=|\d x|^2$, $V=K_g$, and finite absolute
mass means $\int_{\R^2}|V|e^{2u}\,\d x=\int_{\R^2}|K_g|\,\d A_g<\infty$.
No finite-area assumption is imposed.
For positive constant $V$, solutions with $\int_{\R^2}e^{2u}\,\d x<\infty$
represent round spherical metrics in stereographic coordinates~\cite{ChenLi1991}.
For $N=2$ and $F(\xi)^2=\xi^TA\xi$, $A=A^T>0$, one also has $V=K_g$ for
$g=e^{2u}\sum_{i,j=1}^2(A^{-1})_{ij}\,\d x_i\,\d x_j$.
For nonquadratic $F$, we use the conformal invariance of the critical
energy explained in Section~\ref{sec:mass}, without identifying $V$
with an intrinsic Finsler curvature.

\begin{samepage}
The normalization $w=Nu+(N-1)\ln N$ gives $-Q_Nw=Ve^w$.
For positive constant $V$ with $\int_{\RN}Ve^{Nu}\,\d x<\infty$, \cite[Proposition~4.1]{CiraoloLi2024} gives
\[
 \int_{\RN}Ve^{Nu}\,\d x
 =N\kap _N\Big(\frac{N}{N-1}\Big)^{N-1}.
\]
\end{samepage}

Nonconstant prescribed functions lead to different questions. In the Euclidean plane,
Chen--Li~\cite{ChenLi1993} and Cheng--Lin~\cite{ChengLin1997} studied
prescribed-curvature equations and their asymptotic behavior.
Struwe~\cite[Theorem~1.3]{Struwe2020} proved nonexistence for
$V(x)=1+x\cdot Ax$, with $A$ symmetric and negative definite, among
bounded-above entire solutions whose induced metrics have finite area
and finite integrated curvature.
M. Li~\cite[Theorems~1.2 and~1.3]{Li2024} proved nonexistence for
polynomials nonincreasing along rays and for prescribed functions with cone
lower and polynomial upper bounds of the same degree. The comparisons
with our results are given in Remarks~\ref{rem:mass-comparison},
\ref{rem:rigidity-comparison}, and~\ref{rem:polynomial-comparison}.
On complete surfaces with nonnegative Ricci curvature, Ciraolo, Farina,
and Gatti~\cite{CiraoloFarinaGatti2026} obtained classification and rigidity
under asymptotic lower bounds on the solution.

We study bounded-above entire solutions with
$\int_{\RN}|V|e^{Nu}\,\d x<\infty$.
For a variable $V$ that may change sign, this integral does not determine
$\int_{\RN}Ve^{Nu}\,\d x$, and the constant-coefficient mass formula no longer applies.
It is natural to  address the following two questions for~\eqref{maineq}.
\begin{itemize}
\item[(Q1)] Does lower cone growth of $|V|$ give a sharp positive signed-mass
bound without upper growth or monotonicity assumptions?
\item[(Q2)] Which additional growth conditions give annular gradient
convergence and Pohozaev boundary limits sufficient for rigidity under
$x\cdot\nabla V\le0$?
\end{itemize}
The results below answer (Q1) by a capacity argument and (Q2) by
comparison with anisotropic harmonic functions.

We assume $V\in C^1(\RN)$ and consider weak solutions of \eqref{maineq} in the class
\be \label{eq:basic-class}
 u\in C^1_{loc}(\RN)\cap W^{1,N}_{loc}(\RN),\quad
 \sup_{\RN}u<\infty,\quad
 \int_{\RN}|V|e^{Nu}\, \d x<\infty.
\ee
The weak formulation is
\[\int_{\RN} a_N(\nabla u)\cdot\nabla\phi\,\d x=\int_{\RN}Ve^{Nu}\phi\,\d x,\quad  \forall\, \phi\in C_c^\infty(\RN).\]
 For any real-valued function $g$ we write $g_+:=\max\{g,0\}$ and $g_-:=\max\{-g,0\}$.

\begin{defn}
For a growth exponent $\gamma>-N$, the lower cone-growth condition means that there exist a nonempty open cone $\Gamma\subset\RN$ and constants $c_0,R_0>0$ such that
\be \label{eq:cone-growth}
 |V(x)|\ge c_0|x|^\gamma,
 \quad \forall\, x\in\Gamma\ \text{ with }|x|\ge R_0.
\ee
The independent polynomial upper bound used below is
\be \label{eq:polynomial-upper}
 |V(x)|\le C_0(1+|x|)^\si ,
 \quad \forall\, x\in\RN
\ee
for some finite $\si \ge0$ and $C_0>0$. The exponents $\gamma$ and $\si $ may be different.
\end{defn}

The following theorem answers (Q1) under \eqref{eq:cone-growth}.

\begin{thm}\label{thm:threshold}
Assume \eqref{eq:cone-growth}. Then every solution in \eqref{eq:basic-class} has positive signed mass and satisfies
\be \label{eq:threshold}
 M:=\int_{\RN}Ve^{Nu}\, \d x>0,\quad
 \gamma_0:=\Big(\frac{M}{N\kap _N}\Big)^{1/(N-1)}
 \ge1+\frac\gamma N.
\ee
No global upper bound on $V$ or monotonicity assumption is needed.
\end{thm}

\begin{rem}\label{rem:mass-comparison}
(1) For $N=2$ and $F(\xi)=|\xi|$, inequality~\eqref{eq:threshold} reads
$M/(2\pi)\ge1+\gamma/2$. Li~\cite[Lemma~2.1]{Li2024} obtains the same
numerical bound for a logarithmic integral equation under radial
monotonicity and polynomial cone growth with $\gamma\ge2$.
Theorem~\ref{thm:threshold} applies without monotonicity or upper growth,
for every $\gamma>-N$ and for the nonlinear anisotropic operator.
The nonlinear step is the capacity comparison in Section~\ref{sec:mass},
which replaces an additive logarithmic representation.

(2)  Proposition~\ref{prop:equality-example} gives equality for every
$\gamma>-N$, with a sign-changing $V$ and the unbounded logarithmic
remainder in~\eqref{eq:equality-profile}.
Remark~\ref{rem:critical-cone-exponent} proves that $\gamma>-N$ cannot be
weakened if positive signed mass is to follow. The sharpness is within
the class of variable prescribed functions, not positive constants.
\end{rem}

The mass estimate alone does not determine the Pohozaev boundary term. To address (Q2), we prove strong annular gradient convergence for a sequence of dilations under \eqref{eq:polynomial-upper}. Set $\mathcal A=W_2\setminus\overline{W_1}$ and define the Pohozaev vector field
\[
 \mathcal P_u=(x\cdot\nabla u)a_N(\nabla u)-x\Phi_N(\nabla u).
\]
We use $\dashint_A g\,\d x:=|A|^{-1}\int_A g\,\d x$ for integral averages. On a sufficiently regular boundary, $\nu$ denotes the outward Euclidean unit normal and $\mathcal H^{N-1}$ the $(N-1)$-dimensional Hausdorff measure. The boundary limits below will be used in Theorem~\ref{thm:rigidity}.

\begin{thm}\label{thm:annular-pohozaev}
Under the assumptions of Theorem~\ref{thm:threshold}, suppose also that \eqref{eq:polynomial-upper} holds. There exist $R_j\to\infty$ and constants
\[
 \alpha_j=\dashint_{\mathcal A}  (u(R_jy)+\gamma_0\ln  F^0(y) )\, \d y
\]
such that
\[
 u(R_j\,\cdot)-\alpha_j\to -\gamma_0\ln  F^0
  \quad\text{strongly in }\, W^{1,N}_{loc}(\RN\setminus\{0\})\quad\text{ as }\,j\to\infty.
\]
For every $0<a<b<\infty$, as $j\to\infty$, the same sequence satisfies
\begin{equation}\label{eq:auto-annular-gradient-error}
 \int_{W_{bR_j}\setminus W_{aR_j}}
 F(\nabla u+\gamma_0\nabla\ln F^0)^N\,\d x\to0,\quad
 \int_{W_{bR_j}\setminus W_{aR_j}}F(\nabla u)^N\,\d x
 \to N\kap_N\gamma_0^N\ln\frac ba.
\end{equation}
There also exist radii $r_j\in[R_j,2R_j]$ such that, as $j\to\infty$,
\begin{equation}\label{eq:PB-boundary}
 r_j\int_{\pa W_{r_j}}|Ve^{Nu}|\,\d\mathcal H^{N-1}\to0,\quad
 \int_{\pa W_{r_j}}\mathcal P_u\cdot\nu\,\d\mathcal H^{N-1}
 \to(N-1)\kap_N\gamma_0^N.
\end{equation}
\end{thm}

\begin{rem}\label{rem:annular-comparison}
For a positive constant $V$, Ciraolo and
Li~\cite[Proposition~3.1]{CiraoloLi2024} obtain a bounded logarithmic
remainder and
$|x|\,|\nabla(u+\gamma_0\ln F^0)|\to0$ as $|x|\to\infty$.
Theorem~\ref{thm:annular-pohozaev} has a weaker asymptotic conclusion for
a larger class of prescribed functions. Propositions~\ref{prop:equality-example}
and~\ref{prop:annular-subsequence-essential} show why neither a bounded
remainder nor convergence for all radii follows in general.
The exponential comparison and harmonic rigidity results used below are
\cite[Lemma~2.3 and Proposition~6.8]{CiraoloLi2024}; the new argument
establishes compactness for the signed equation. The boundary limits do
not quantize $M$, as the masses in~\eqref{eq:equality-mass} vary continuously.
\end{rem}

Theorem~\ref{thm:rigidity} below uses the limits~\eqref{eq:PB-boundary},
which we formulate separately from~\eqref{eq:polynomial-upper}.
\begin{defn}\label{def:pohozaev-infinity}
A solution satisfying the hypotheses of Theorem~\ref{thm:threshold}
satisfies the \emph{Pohozaev condition at infinity} if there exists a
sequence $r_j\to\infty$ for which both limits in~\eqref{eq:PB-boundary}
hold, with $\gamma_0$ given by~\eqref{eq:threshold}. The same radii must
satisfy both limits; a subsequence may be chosen strictly increasing.
\end{defn}

The limits \eqref{eq:PB-boundary} give the global Pohozaev identity. Writing $D(x):=x\cdot\nabla V(x)$ and using $D_+=\max\{D,0\}$, $D_-=\max\{-D,0\}$, Proposition~\ref{prop:global-pohozaev} shows, whenever at least one of $D_+e^{Nu}$ and $D_-e^{Nu}$ is integrable, that both are integrable and
\be \label{eq:pohozaev-intro}
 \gamma_0^{N-1} ((N-1)\gamma_0-N )
 =\frac1{N\kap _N}\int_{\RN}D(x)e^{Nu}\, \d x.
\ee
In particular, a radial derivative of one sign needs no separate integrability assumption. Combining \eqref{eq:pohozaev-intro} with \eqref{eq:threshold} gives Theorem~\ref{thm:rigidity}. Here $x\cdot\nabla V\le0$ means
\[
 \frac{\d}{\d r}V(r\theta)
 =\frac{(r\theta)\cdot\nabla V(r\theta)}{r}\le0,
 \quad r>0,\ \theta\in\mathbb S^{N-1};
\]
no radial symmetry of $V$ is assumed.

\begin{thm}\label{thm:rigidity}
Assume \eqref{eq:cone-growth} with
\[
 \gamma\ge\frac{N}{N-1}.
\]
Assume in addition that
\[
 x\cdot\nabla V(x)\le0, \quad \forall\, x\in\RN.
\]
Then no solution in \eqref{eq:basic-class} satisfies the Pohozaev condition at infinity (Definition~\ref{def:pohozaev-infinity}).
\end{thm}

Theorems~\ref{thm:annular-pohozaev} and~\ref{thm:rigidity} give Corollary~\ref{cor:unconditional-rigidity} without an independent boundary assumption.

\begin{cor}\label{cor:unconditional-rigidity}
Under the assumptions on $V$ in Theorem~\ref{thm:rigidity}, if \eqref{eq:polynomial-upper} holds, equation \eqref{maineq} has no solution in \eqref{eq:basic-class}.
\end{cor}

\begin{rem}\label{rem:rigidity-comparison}
Proposition~\ref{prop:monotone-sharpness} proves that the exponent
$N/(N-1)$ cannot be decreased: for every $0\le\gamma<N/(N-1)$, it gives
a solution with $x\cdot\nabla V<0$ for $x\ne0$ and
$c|x|^\gamma\le|V(x)|\le C|x|^\gamma$ near infinity.
In the Euclidean plane, Corollary~\ref{cor:unconditional-rigidity} has
the same lower exponent $\gamma\ge2$ as Li~\cite[Theorem~1.3]{Li2024},
but does not require $\si=\gamma$ in~\eqref{eq:polynomial-upper}. The other extensions
are the dimension and the anisotropy.

Proposition~\ref{prop:larger-upper-growth} allows the polynomial upper
bound to be replaced by
\be\label{eq:larger-upper-growth}
 \ln(e+|V(x)|)\le C[\ln(e+|x|)]^{N/(N-1)},\quad x\in\RN.
\ee
For the full annular conclusion, the power $2$ when $N=2$ is optimal
by Remark~\ref{rem:upper-growth-planar}. Its optimality for $N>2$, or
for nonexistence rather than annular convergence, is not established here.
\end{rem}

Without~\eqref{eq:polynomial-upper} or~\eqref{eq:larger-upper-growth}, the limits~\eqref{eq:PB-boundary} also follow from finite exterior gradient error,
\[
 \cG_\infty^{(N)}(u)
 =\int_{\{F^0>1\}}F(\nabla u+\gamma_0\nabla\ln  F^0)^N\, \d x<\infty.
\]
Proposition~\ref{prop:finite-gradient-error} derives the Pohozaev condition at infinity from $\cG_\infty^{(N)}(u)<\infty$, independently of Theorem~\ref{thm:annular-pohozaev}. Combining it with Theorem~\ref{thm:rigidity} gives the following corollary.

\begin{cor}\label{cor:finite-gradient-error}
Under the assumptions of Theorem~\ref{thm:rigidity}, there is no solution in \eqref{eq:basic-class} with finite exterior gradient error.
\end{cor}

A concrete consequence concerns polynomial prescribed functions. If a nonconstant polynomial $P$ satisfies $x\cdot\nabla P\le0$, its leading degree $d$ is even and at least two, and \eqref{eq:cone-growth}--\eqref{eq:polynomial-upper} hold with $\gamma=\si=d$; see Section~\ref{sec:pohozaev}.

\begin{cor}\label{cor:polynomial}
Let $P$ be a nonconstant polynomial on $\RN$ such that $x\cdot\nabla P\le0$. Then $-Q_Nu=Pe^{Nu}$ has no solution in \eqref{eq:basic-class}. In particular, this applies to $V=1+P_m$, where $P_m\not\equiv0$ is a nonpositive homogeneous polynomial of degree $m\ge2$.
\end{cor}

\begin{rem}\label{rem:polynomial-comparison}
The Euclidean planar case of Corollary~\ref{cor:polynomial} is Li's Theorem~1.2 in~\cite{Li2024}; the negative definite quadratic case was proved by Struwe~\cite[Theorem~1.3]{Struwe2020}. The extension in Corollary~\ref{cor:polynomial} is to all critical dimensions and fixed anisotropies, not a new planar polynomial theorem.
\end{rem}

The mass estimate also has a planar geometric consequence. In the Euclidean plane, $V$ is the Gaussian curvature of $g=e^{2u}|\d x|^2$, and the signed mass is its total curvature. Theorem~\ref{thm:threshold} and the Cohn--Vossen inequality~\cite[Satz~6]{CohnVossen1935} imply that this metric cannot be complete when $\gamma>0$; the threshold $\gamma=0$ is sharp. The argument and the examples are given in Section~\ref{subsec:curvature-sharpness}.

The difficulty in (Q1) is to estimate $\int_{\RN}Ve^{Nu}\,\d x$ directly:
except in the quadratic planar case, the operator is nonlinear, so
potentials for the positive and negative parts of $Ve^{Nu}$ cannot be added. For $v=b_u-u\ge1$, where $b_u=\sup_{\RN}u+1$, the logarithmic energy
identity~\eqref{eq:automatic-log-energy} gives
\eqref{eq:multiplicative-convergence}, and cone growth gives
\eqref{eq:cone-normalization-lower}. Convergence in $L^p$ alone does not
exclude a thin connected sublevel set crossing an annulus.
Lemma~\ref{lem:annular-capacity} and the small critical energy in
\eqref{eq:log-convergence} exclude such a component, giving an admissible
capacity truncation without regularity or convexity of its boundary.
The signed identity~\eqref{eq:sublevel-boundary} and exact Wulff
capacity~\eqref{eq:wulff-capacity} then yield~\eqref{eq:annular-scale-bounds}.

For (Q2), we compare $v$ with $F$-$N$-harmonic functions on annuli of
increasing logarithmic width. Exponential comparison
\cite[Lemma~2.3]{CiraoloLi2024} and strong monotonicity give
\eqref{eq:comparison-gradient} without uniformly elliptic linearization;
see also~\cite{BrezisMerle1991} and~\cite[Proposition~2.1]{Esposito2018}.
The upper growth bound occurs only in its logarithmic factor.
The remaining difficulty is $\om_k(h_j)^{N-2}$ in
\eqref{eq:boundary-master-N}: division by $\om_k(h_j)^{N-1}$ as in
\eqref{eq:relative-boundary-N} and Lemma~\ref{lem:annular-growth}
give a uniform oscillation bound before the boundary integrals converge
to $M$. This order determines the harmonic limit without a bounded
pointwise remainder; see Remark~\ref{rem:annular-comparison}.

No third derivative of $F$ is used, and no necessity is claimed for
evenness or uniform convexity. Theorems~\ref{thm:threshold}
and~\ref{thm:annular-pohozaev} do not differentiate $V$: local boundedness
suffices, with their growth assumptions unchanged. The Pohozaev identities
use first derivatives, so $V\in C^1(\RN)$ is assumed throughout.
Remark~\ref{rem:measurable-sector} replaces the open cone by a measurable
angular sector of positive measure.

The paper is organized as follows. We prove Theorem~\ref{thm:threshold}
in Section~\ref{sec:mass} and Theorem~\ref{thm:annular-pohozaev} in
Section~\ref{sec:annular-compactness}. The Pohozaev identities and the
proof of Theorem~\ref{thm:rigidity}, together with its consequences,
are given in Section~\ref{sec:pohozaev}. We conclude in
Section~\ref{sec:applications} with the sharpness examples and the
planar completeness results.

\section{The signed-mass threshold and logarithmic compactness}\label{sec:mass}

We first derive the geometric identities associated with \eqref{ass:F} and then prove an exact logarithmic energy identity for sign-changing right-hand sides.

Since $F$ is continuous and positive on the Euclidean unit sphere, there exist constants $0<c_F\le C_F<\infty$ such that
\[
 c_F|\xi|\le F(\xi)\le C_F|\xi|,
 \quad
 C_F^{-1}|x|\le F^0(x)\le c_F^{-1}|x|.
\]
The dual estimates follow from the definition of $F^0$ and give $B_{c_Fr}\subset W_r\subset B_{C_Fr}$. Constants $c_{F,N},C_{F,N}>0$ depend only on $F,N$ and may change from line to line. We write $A\asymp B$ when $A$ and $B$ are comparable by positive constants.

By \eqref{ass:F}, $a_N=D\Phi_N$ is odd and homogeneous of degree $N-1$.
Euler's identity for $\Phi_N=F^N/N$ gives
\[
 a_N(\xi)\cdot\xi=F(\xi)^N=N\Phi_N(\xi).
\]
Standard consequences of uniform convexity are
\begin{align}
 |a_N(\xi)|&\le C_F|\xi|^{N-1},\label{eq:a-growth}\\
 (a_N(\xi)-a_N(\eta))\cdot(\xi-\eta)
 &\ge c_F(|\xi|+|\eta|)^{N-2}|\xi-\eta|^2
 \ge c_F|\xi-\eta|^N,\label{eq:a-monotonicity}\\
 |a_N(\xi)-a_N(\eta)|
 &\le C_F(|\xi|+|\eta|)^{N-2}|\xi-\eta|.
 \label{eq:a-difference}
\end{align}
The Hessian $D^2\Phi_N(\zeta)$ is comparable to $|\zeta|^{N-2}I$ for $\zeta\ne0$, by homogeneity, norm equivalence, and \eqref{ass:F}. Integration on $\eta+t(\xi-\eta)$, $0<t<1$, gives \eqref{eq:a-monotonicity} and \eqref{eq:a-difference}. The last inequality in \eqref{eq:a-monotonicity} follows from $|\xi|+|\eta|\ge|\xi-\eta|$.  For $N=2$, $a_2$ is globally Lipschitz and strongly monotone. For $N>2$, Section~\ref{sec:annular-compactness} uses \eqref{eq:a-monotonicity}--\eqref{eq:a-difference} without linearization.

\subsection{Wulff geometry and logarithmic compactness}

We first state the local regularity estimate used in the compactness arguments. Its constants depend on the fixed anisotropy, not on a positive lower bound for the gradient.

\begin{lem}\label{lem:local-C1-compactness}
Let $B_{2r}(x_0)\Subset D$ and let $w\in W^{1,N}_{loc}(D)$ be a weak solution of $\diver a_N(\nabla w)=g$ in $D$. If $w$ and $g$ are bounded on $B_{2r}(x_0)$, then $w\in C^{1,\alpha}(B_r(x_0))$ for some $\alpha\in(0,1)$. The local $C^{1,\alpha}$ bound depends only on $F$, $N$, $r$, and the $L^\infty$ norms of $w$ and $g$ on $B_{2r}(x_0)$. In particular, uniform such bounds imply local precompactness in $C^1$.
\end{lem}

\begin{proof}
Condition~\eqref{ass:F}, homogeneity, and compactness of the unit sphere give, for $\xi\ne0$,
\[
 c|\xi|^{N-2}|\zeta|^2
 \le \langle Da_N(\xi)\zeta,\zeta\rangle
 \le C|\xi|^{N-2}|\zeta|^2,
 \quad |Da_N(\xi)|\le C|\xi|^{N-2}.
\]
For $N>2$, $Da_N$ extends continuously by zero at the origin. The standard scalar $N$-growth regularity estimates therefore apply; see \cite{DiBenedetto1983,Tolksdorf1984}.

For $N=2$, the map $a_2$ is globally Lipschitz and strongly monotone, although it need not be differentiable at zero. Difference quotients in the weak equation give $w\in W^{2,2}_{loc}$. Only $g\in L^2_{loc}(D)$ is used: discrete integration by parts places the difference quotient on the test function, and ellipticity estimates the resulting term. The Sobolev chain rule then yields
\[
 \diver(B(x)\nabla\pa _iw)=\pa _i g
       =\diver(g e_i),
\]
where $e_i$ is the $i$th coordinate vector and $B(x)=Da_2(\nabla w(x))$ when $\nabla w(x)\ne0$. On the zero-gradient set choose any fixed matrix with the same ellipticity bounds. This choice does not change the identity, since $D^2w=0$ almost everywhere on that set. The resulting $B$ is measurable, bounded, and uniformly elliptic. The inhomogeneous De Giorgi estimates, with the bounded vector field $g e_i$ on the right, give local H\"older bounds for $\pa _iw$; see \cite[Chapter~8]{GilbargTrudinger2001}. Caccioppoli's inequality controls the local $L^2$ norm of the gradient in terms of the stated bounds. The uniform $C^{1,\alpha}$ estimate follows also in dimension two; local $C^1$ precompactness follows from Arzel\`a--Ascoli.
\end{proof}

For $x\ne0$, strict convexity gives a unique maximizer $\xi_x$ of
$x\cdot\xi$ on $\{F=1\}$, and differentiation of the support function
gives $\xi_x=\nabla F^0(x)$. The multiplier in
$x=\lambda F_\xi(\xi_x)$ equals $F^0(x)$ by Euler's identity.
Thus
\be \label{eq:polar-identities}
 F(\nabla F^0)=1,\quad
 F_\xi(\nabla F^0)=\frac{x}{F^0},\quad
 x\cdot\nabla F^0=F^0,\quad
 a_N(\nabla F^0)=\frac{x}{F^0}.
\ee
The regularity of $F^0$ also follows from \eqref{ass:F}. Indeed, the Legendre map $D(F^2/2)$ has positive definite derivative away from zero, and its inverse is $D((F^0)^2/2)$. The inverse function theorem gives $F^0\in C^2(\RN\setminus\{0\})$. Euler's identity implies $\nabla F^0\ne0$ there, so every Wulff sphere has a $C^2$ boundary. No smoothness of either norm at zero is used.

Geometrically, $F$ and $F^0$ measure covectors and tangent vectors, respectively;
the sets $W_r$ are the balls of the corresponding flat normed geometry.
For a conformal change $F^0\mapsto e^uF^0$, the dual norm is $e^{-u}F$ and
the volume density, normalized by Lebesgue measure for the undeformed norm,
is $e^{Nu}\,\d x$. Hence, for $\varphi\in C_c^1(\RN)$,
\[
 \int_{\RN}(e^{-u}F(\nabla\varphi))^N e^{Nu}\,\d x
 =\int_{\RN}F(\nabla\varphi)^N\,\d x.
\]
The $N$-energy is therefore unchanged by the conformal rescaling.
Euclidean inversion, however, need not preserve a general anisotropic operator.

\begin{lem}
For every nonnegative measurable $g$,
\be \label{eq:polar-integration}
 \int_{\RN}g(F^0(x))\, \d x
 =N\kap _N\int_0^\infty g(r)r^{N-1}\, \d r.
\ee
Moreover,
\be \label{eq:minkowski}
 \int_{\pa  \W_R}x\cdot\nu\, \d\mathcal H^{N-1}
 =N\kap _NR^N,
\ee
and $\ln  F^0$ is $\QN$-harmonic on $\RN\setminus\{0\}$, with
\be \label{eq:log-boundary}
 \int_{\pa  \W_R}a_N(\nabla\ln  F^0)\cdot\nu\,\d\mathcal H^{N-1}
 =N\kap _N.
\ee
\end{lem}

\begin{proof}
Since $|\W_r|=\kap _Nr^N$, integration with respect to the distribution function of $F^0$ gives \eqref{eq:polar-integration}. The divergence theorem applied to $x$ gives \eqref{eq:minkowski}. By \eqref{eq:polar-identities} and homogeneity,
\[
 a_N(\nabla\ln  F^0)=\frac{x}{(F^0)^N},\quad
 \diver\Big(\frac{x}{(F^0)^N}\Big)
 =\frac{N}{(F^0)^N}-\frac{N\,x\cdot\nabla F^0}{(F^0)^{N+1}}=0
 \quad (x\ne0).
\]
On $\pa \W_R$, integration and \eqref{eq:minkowski} yield
$\int_{\pa \W_R}a_N(\nabla\ln  F^0)\cdot\nu\,\d\mathcal H^{N-1}
=R^{-N}N\kap _NR^N=N\kap _N$.
\end{proof}

On $\pa W_R$, \eqref{eq:polar-identities} gives $x\cdot\nu=R F(\nu)$. Thus \eqref{eq:minkowski} identifies the anisotropic perimeter:
\[
 P_F(W_R):=\int_{\pa W_R}F(\nu)\,\d\mathcal H^{N-1}
 =N\kap _N R^{N-1}.
\]
In particular, the constant $N\kap _N$ in \eqref{eq:log-boundary} is the anisotropic perimeter of $W_1$.

The same logarithmic profile also minimizes the $N$-energy in a Wulff annulus. Here $E\Subset\Om $ means that the closure of $E$ is compactly contained in $\Om $. For a bounded Lipschitz domain $\Om $ with $\overline{\W_r}\subset\Om $, define
\[
 \capFN(\W_r,\Om )
 :=\inf\Big\{\int_\Om  F(\nabla\varphi)^N\,\d x:
 \varphi\in W^{1,N}(\Om ),\,  \varphi=0\text{ a.e. on }\W_r,
 \varphi-1\in W^{1,N}_0(\Om )
\Big\}.
\]
The last condition specifies the outer trace. This gives the exact capacity formula and its monotonicity with respect to bounded Lipschitz outer domains. The capacity comparison below uses only $C^2$ Wulff annuli.
\begin{lem}
For $0<r<R$,
\be \label{eq:wulff-capacity}
 \capFN(\W_r,\W_R)
 =\frac{N\kap _N}{[\ln (R/r)]^{N-1}}.
\ee
If $\W_r\Subset\Om _1\subset\Om _2$ and both outer domains are bounded and Lipschitz, then
\be \label{eq:capacity-monotonicity}
 \capFN(\W_r,\Om _1)\ge\capFN(\W_r,\Om _2).
\ee
\end{lem}

\begin{proof}
In $\W_R\setminus\overline{\W_r}$ set
\[
 \varphi_0(x)=\frac{\ln (F^0(x)/r)}{\ln (R/r)}.
\]
The function $\varphi_0$ is $Q_N$-harmonic and has the prescribed traces. For any admissible $\varphi$, convexity gives
\begin{align*}
\int_{\W_R\setminus\overline{\W_r}}\Phi_N(\nabla\varphi)\,\d x
 &\ge\int_{\W_R\setminus\overline{\W_r}}\Phi_N(\nabla\varphi_0)\,\d x
 +\int_{\W_R\setminus\overline{\W_r}} a_N(\nabla\varphi_0)\cdot\nabla(\varphi-\varphi_0)\,\d x\\
 &=\int_{\W_R\setminus\overline{\W_r}}\Phi_N(\nabla\varphi_0)\,\d x.
\end{align*}
Thus \eqref{eq:polar-integration} and $F(\nabla F^0)=1$ give
\[
 \capFN(\W_r,\W_R)
 =\frac{N\kap _N}{[\ln (R/r)]^N}\int_r^R\frac{\d s}{s}
 =\frac{N\kap _N}{[\ln (R/r)]^{N-1}}.
\]
Extension by one from $\Om _1$ to $\Om _2$ preserves admissibility and energy, proving \eqref{eq:capacity-monotonicity}.
\end{proof}

The coarea formula gives small boundary integrals for integrable functions.

\begin{lem}\label{lem:good-radii}
If $g\ge0$ belongs to $L^1(\{F^0>1\})$, there exist $R_j\to\infty$ such that
\[
 R_j\int_{\pa  \W_{R_j}}g\,\d\mathcal H^{N-1}\to0\quad\text{ as }\, j\to\infty.
\]
\end{lem}

\begin{proof}
Coarea, the uniform bound for $|\nabla F^0|$, and averaging over $[2^j,2^{j+1}]$ give a radius $R_j$ in that interval such that
\[
 R_j\int_{\pa \W_{R_j}}g\,\d\mathcal H^{N-1}
 \le 2C_F\int_{\{2^j<F^0<2^{j+1}\}}g\,\d x\to0.
\]
The last limit follows from the integrability of $g$.
\end{proof}

Set $f:=Ve^{Nu}$ and
\[
 b_u=\sup_{\RN}u+1,\quad v=b_u-u\ge1.
\]
The evenness $F(-\xi)=F(\xi)$ in \eqref{ass:F} gives
$\Phi_N(-\xi)=\Phi_N(\xi)$ and hence $a_N(-\xi)=-a_N(\xi)$.
As $\nabla v=-\nabla u$, equation~\eqref{maineq} becomes
\be \label{eq:v-equation}
 \diver a_N(\nabla v)=f\quad\text{ in }\, \RN.
\ee
Testing \eqref{eq:v-equation} with a negative power of $v$ gives an exact logarithmic energy identity.

\begin{prop}
Let $v\in C^1_{loc}(\RN)\cap W^{1,N}_{loc}(\RN)$ be a weak solution of~\eqref{eq:v-equation}, with $v\ge1$ and $f\in L^1(\RN)$. Then
\be \label{eq:automatic-log-energy}
 (N-1)\int_{\RN}F(\nabla\ln  v)^N\, \d x
 =\int_{\RN}\frac{f}{v^{N-1}}\, \d x
 \le\int_{\RN}f_+\, \d x<\infty.
\ee
\end{prop}

\begin{proof}
Let $0\le\eta\le1$ be compactly supported and use
\[
 \phi=\frac{\eta^N}{v^{N-1}}
\]
as a test function.  Since
\[
 \nabla\phi
 =N\eta^{N-1}v^{1-N}\nabla\eta
 -(N-1)\eta^Nv^{-N}\nabla v
\]
and $a_N(\nabla v)\cdot\nabla v=F(\nabla v)^N$, the weak identity becomes
\be\label{eq:log-test-identity}
 (N-1)\int_{\RN}\eta^N F(\nabla\ln v)^N\,\d x
 =\int_{\RN}f\frac{\eta^N}{v^{N-1}}\,\d x
 +N\int_{\RN}\eta^{N-1}v^{1-N}a_N(\nabla v)\cdot\nabla\eta\,\d x.
\ee
As $0\le\eta\le1$ and $v\ge1$, the first term on the right of \eqref{eq:log-test-identity} is at most $\int_{\RN}f_+\,\d x$.  Furthermore, by \eqref{eq:a-growth} and norm equivalence,
\[
 v^{1-N}|a_N(\nabla v)|
 \le C_F F(\nabla\ln  v)^{N-1}.
\]
Young's inequality in the form $AB\le \var  A^{N/(N-1)}+C_\var  B^N$ therefore gives
\[
 \int_{\RN}\eta^N F(\nabla\ln v)^N\,\d x
 \le C_{F,N}\int_{\RN}f_+\,\d x
 +C_{F,N}\int_{\RN}F(\nabla\eta)^N\,\d x.
\]

It remains to let the support of $\eta$ tend to infinity with critical energy tending to zero.  For $R>e$, start from the radial cutoff
\[
 \eta_R(x)=
\left\{\begin{aligned}
    &1,&&|x|\le R,\\
 &2-\dfrac{\ln |x|}{\ln  R},&&R<|x|<R^2,\\
& 0,&&|x|\ge R^2,
\end{aligned}\right.
\]
and smooth it near the interfaces.  On the transition annulus, $|\nabla\eta_R|\le C/(|x|\ln  R)$.  Hence
\be \label{eq:log-cutoff}
 \int_{\RN}F(\nabla\eta_R)^N\, \d x
 \le \frac{C_{F,N}}{(\ln  R)^N}
 \int_R^{R^2}\frac{\d r}{r}
 \le\frac{C_{F,N}}{(\ln R)^{N-1}}\to0\quad\text{ as }\, R\to\infty.
\ee
Fatou's lemma gives $F(\nabla\ln v)^N\in L^1(\RN)$. The last term in \eqref{eq:log-test-identity}, with $\eta=\eta_R$, is bounded in absolute value by
\[
 C\Big(\int_{\operatorname{supp}\nabla\eta_R}
 F(\nabla\ln  v)^N\, \d x\Big)^{(N-1)/N}
 \Big(\int_{\RN} F(\nabla\eta_R)^N\, \d x\Big)^{1/N},
\]
which tends to zero as $R\to\infty$. The estimates $|f|\eta_R^Nv^{1-N}\le|f|$ and $\eta_R^NF(\nabla\ln v)^N\le F(\nabla\ln v)^N$ justify dominated convergence in \eqref{eq:log-test-identity}, proving \eqref{eq:automatic-log-energy}.
\end{proof}

Theorem~\ref{thm:threshold} can fail when $\sup_{\RN}u=\infty$, even if the remaining assumptions hold, as the following example shows.

\begin{exa}
In dimension two with $F(\xi)=|\xi|$, set $r=|x|$ and
\[
 u(x)=x_1+\ln (1+r^2)+2e^{-r^2},\quad
 f(x)=-\frac4{(1+r^2)^2}+8(1-r^2)e^{-r^2},\quad
 V(x)=f(x)e^{-2u(x)}.
\]
Then $u,V$ are smooth, $-\Delta u=Ve^{2u}=f$, and $f\in L^1(\R^2)$. $V$ changes sign: $f(0)=4$, whereas $f<0$ for $r\ge1$. Integration gives
\[
 \int_{\R^2}Ve^{2u}\,\d x=-4\pi.
\]
On the cone $x_1<-|x|/2$, one has $|V(x)|\ge c e^{|x|}(1+|x|^2)^{-4}$ for large $|x|$, so the lower cone bound holds for every finite exponent $\gamma$. However, $\sup_{\R^2}u=\infty$. This example does not satisfy \eqref{eq:basic-class}; it shows why that upper bound cannot be omitted.
\end{exa}

\begin{prop}\label{prop:multiplicative-convergence}
Let $A\Subset\RN\setminus\{0\}$ be bounded, connected, and Lipschitz. Set
\[
 m_A(R):=\dashint_A \ln  v(Ry)\, \d y,
 \quad
 c_A(R):=e^{m_A(R)}.
\]
Then, as $R\to\infty$,
\be \label{eq:log-convergence}
 \ln \frac{v(R\,\cdot)}{c_A(R)}\to 0
 \quad\text{ in }\, W^{1,N}(A),
\ee
and, for every $1\le p<\infty$,
\be \label{eq:multiplicative-convergence}
 \frac{v(R\,\cdot)}{c_A(R)}\to 1
 \quad\text{ in }\, L^p(A).
\ee
\end{prop}

\begin{proof}
Set
\[
 h_R(y)=\ln  v(Ry)-m_A(R).
\]
Since $\nabla_y\ln  v(Ry)=R\nabla_x\ln  v(Ry)$ and the exponent equals the dimension,
\[
 \int_A F(\nabla h_R)^N\, \d y
 =\int_{RA}F(\nabla\ln  v)^N\, \d x.
\]
Since $\inf_{y\in A}|Ry|\to\infty$ as $R\to\infty$, the right-hand side tends to zero by \eqref{eq:automatic-log-energy}.  The normalization $\dashint_A h_R\,\d y=0$ and Poincar\'e's inequality give $h_R\to0$ in $W^{1,N}(A)$, proving \eqref{eq:log-convergence}.

The passage from $h_R$ to $e^{h_R}$ is not automatic from Sobolev convergence at the critical exponent.  Use a bounded Sobolev extension from the fixed Lipschitz domain $A$, followed by a cutoff in a fixed containing ball. The resulting zero-trace extensions tend to zero in $W^{1,N}$. Applying the critical Moser--Trudinger inequality~\cite{Trudinger1967} to these extensions gives a constant $\tau_A>0$ such that
\[
 \sup_{R\ge R_0}\int_A
 \exp(\tau_A|h_R|^{N/(N-1)})\, \d y<\infty
\]
for all sufficiently large $R$. For fixed $1\le p<\infty$, Young's
inequality gives $\sup_{R\ge R_0}\|e^{|h_R|}\|_{L^{2p}(A)}<\infty$.
The Sobolev embedding gives $\|h_R\|_{L^{2p}(A)}\to0$. Therefore
$|e^s-1|\le |s|e^{|s|}$ and H\"older's inequality yield
\[
 \|e^{h_R}-1\|_{L^p(A)}
 \le\|h_R\|_{L^{2p}(A)}\|e^{|h_R|}\|_{L^{2p}(A)}
 \to0\quad\text{ as }\,R\to\infty,
\]
proving~\eqref{eq:multiplicative-convergence}.
\end{proof}

\begin{rem}\label{rem:annular-normalizations}
If two fixed annuli $A$ and $E$ are contained in a common bounded connected Lipschitz set away from the origin, apply Proposition~\ref{prop:multiplicative-convergence} on the larger set.  The means of $\ln  v(R\cdot)$ over $A$ and $E$ then differ by $o(1)$, and therefore
\[
 \frac{c_A(R)}{c_E(R)}\to1\quad\text{ as }\, R\to\infty.
\]
The leading multiplicative scale is therefore independent of the fixed annulus.
\end{rem}

\begin{rem}
Convergence~\eqref{eq:multiplicative-convergence} is relative to $c_A(R)$. If $c_A(R)\asymp\ln R$, it gives only $v(R\,\cdot)-c_A(R)=o(\ln R)$ in $L^p(A)$. The gradient estimates for the Pohozaev boundary term are proved separately in Section~\ref{sec:annular-compactness}.
\end{rem}


We now combine \eqref{eq:multiplicative-convergence} with \eqref{eq:cone-growth}, where $\gamma>-N$, to prove Theorem~\ref{thm:threshold}. No upper growth assumption on $V$ is used in this section.

Fix
\[
 A_0:=\{y\in \RN:1<|y|<2\},\quad
 A_1:=\{y\in \RN:\tfrac12<|y|<3\},
\]
and write
\[
 m(R):=\dashint_{A_1}\ln  v(Ry)\, \d y,
 \quad c(R):=e^{m(R)}.
\]
The normalization is positive, and Proposition~\ref{prop:multiplicative-convergence} identifies $c(R)$ with the average of $v(R\cdot)$ on any fixed subset of $A_1$ of positive measure, up to a relative error tending to zero.

\subsection{Proof of Theorem~\ref{thm:threshold}}

We first use the cone hypothesis to obtain the lower bound for the annular normalization. The growth of $|V|$ and the exponential nonlinearity determine the logarithmic rate.

\begin{prop}\label{prop:cone-forcing}
Under \eqref{eq:cone-growth},
\be \label{eq:cone-normalization-lower}
 \liminf_{R\to\infty}\frac{c(R)}{\ln  R}
 \ge 1+\frac\gamma N>0.
\ee
\end{prop}

\begin{proof}
Choose a ball $E\Subset A_1\cap\Gamma$. Condition~\eqref{eq:cone-growth} gives $|V(Ry)|\ge cR^\gamma$ on $E$ for all sufficiently large $R$, also when $\gamma<0$.  With $f=Ve^{Nu}$ and $u=b_u-v$, it follows that
\[
 \delta_R:=\int_{RE}|f|\, \d x
 \ge cR^{N+\gamma}\int_Ee^{-Nv(Ry)}\, \d y.
\]
Since $\inf_{y\in E}|Ry|\to\infty$, absolute integrability gives $\delta_R\to0$ as $R\to\infty$.  Jensen's inequality yields
\[
 \delta_R\ge cR^{N+\gamma}
 \exp\Big(-\frac N{|E|}\int_Ev(Ry)\, \d y\Big),
\]
where the positive constant $c$ may change from line to line.  Consequently,
\[
 \dashint_E v(Ry)\, \d y
 \ge\Big(1+\frac\gamma N\Big)\ln  R
 +\frac1N\ln \frac c{\delta_R}.
\]
By Proposition~\ref{prop:multiplicative-convergence} on $A_1$, the left-hand side is $c(R)(1+o(1))$.  For large $R$, the final term on the right is nonnegative.  Division by $\ln R$ proves \eqref{eq:cone-normalization-lower}.
\end{proof}

Proposition~\ref{prop:cone-forcing} uses only $|V|$ on the cone. For the complementary upper estimate, integration over sublevel sets of $v$ gives an identity involving the signed function $f$. Define
\[
 G(t):=\int_{\{v<t\}}f\, \d x,\quad  t>1.
\]
The signed quantity $G(t)$ is a nonnegative boundary integral for almost every level.

\begin{prop}\label{prop:sublevel-boundary}
For almost every $t>1$,
\be \label{eq:sublevel-boundary}
 G(t)=\int_{\{v=t\}}
 \frac{F(\nabla v)^N}{|\nabla v|}\,\d\mathcal H^{N-1}\ge0,
\ee
where the integrand is defined to be zero when $\nabla v=0$.  Moreover,
\be \label{eq:G-limit}
 G(t)\to  M:=\int_{\RN}f\, \d x
 \quad\text{ as }\, t\to\infty,
\ee
and, for every $1<s<t$,
\be \label{eq:layer-energy}
 \int_{\{s<v<t\}}F(\nabla v)^N\, \d x
 =\int_s^tG(\tau)\,\d\tau.
\ee
In particular, $M\ge0$.
\end{prop}

\begin{proof}
For $t>1$ and $\var >0$, let $q_{t,\var }$ equal one on $(-\infty,t]$, decrease linearly to zero on $[t,t+\var ]$, and vanish on $[t+\var ,\infty)$.  Let $\eta_R$ be the logarithmic cutoffs used in \eqref{eq:log-cutoff}.  Testing $\diver a_N(\nabla v)=f$ with $\eta_Rq_{t,\var }(v)$ gives
\be \label{eq:boundary-test}
 \int_{\RN} f\eta_Rq_{t,\var }(v)\, \d x
 =\frac1\var \int_{\{t<v<t+\var \}}
       \eta_RF(\nabla v)^N\, \d x
 -\int_{\RN} q_{t,\var }(v)a_N(\nabla v)\cdot\nabla\eta_R\, \d x.
\ee
The test is justified by the Sobolev chain rule and compactly supported approximation.  On the support of $q_{t,\var }(v)$,
\[
 |a_N(\nabla v)|
 \le C(t+\var )^{N-1}F(\nabla\ln  v)^{N-1}.
\]
H\"older's inequality therefore bounds the last term in \eqref{eq:boundary-test} by
\[
 C(t+\var )^{N-1}
 \Big(\int_{\operatorname{supp}\nabla\eta_R}
               F(\nabla\ln  v)^N\, \d x\Big)^{(N-1)/N}
 \Big(\int_{\RN} F(\nabla\eta_R)^N\, \d x\Big)^{1/N},
\]
which tends to zero as $R\to\infty$. On $\{t<v<t+\var\}$, $F(\nabla v)^N\le(t+\var)^NF(\nabla\ln v)^N\in L^1(\RN)$. Dominated convergence in \eqref{eq:boundary-test} gives
\be \label{eq:epsilon-boundary}
 \int_{\RN} fq_{t,\var }(v)\, \d x
 =\frac1\var \int_{\{t<v<t+\var \}}
                  F(\nabla v)^N\, \d x.
\ee

The measure on $\R$ defined by $\mu(B)=\int_{\{v\in B\}}|f|\, \d x$ is finite and has at most countably many atoms.  Thus, for almost every $t$, the left-hand side of \eqref{eq:epsilon-boundary} converges to $G(t)$ as $\var \downarrow0$.  On the right, coarea followed by Lebesgue differentiation gives \eqref{eq:sublevel-boundary}.  No regular-value assumption is needed.

Since $v$ is finite everywhere and $f\in L^1(\RN)$, dominated convergence gives \eqref{eq:G-limit}.  Integrating \eqref{eq:sublevel-boundary} in $t$ and applying coarea once more proves \eqref{eq:layer-energy}.  The limit in \eqref{eq:G-limit}, together with the almost-everywhere nonnegativity of $G$, implies $M\ge0$.
\end{proof}

Multiplicative convergence in $L^p$ does not exclude a connected low-value set joining the two boundary spheres. Lemma~\ref{lem:annular-capacity} gives a uniform positive capacity bound for such sets; combined with \eqref{eq:log-convergence}, it yields a bounded sublevel component and the truncation in \eqref{eq:capacity-chain}.

For $E\subset A_1$, let $\mathcal C(E)$ consist of the functions
$\zeta\in W^{1,N}(A_1)$ that satisfy $\zeta\ge1$ almost everywhere on an open neighborhood of $E$. We define the relative $(1,N)$-Sobolev capacity of $E$ in $A_1$ by
\[
 \SobCap_{1,N}(E;A_1)
 :=\inf_{\zeta\in\mathcal C(E)}
 \int_{A_1}(|\nabla\zeta|^N+|\zeta|^N)\, \d y.
\]

\begin{lem}\label{lem:annular-capacity}
There exists $c_{\mathrm{cap}}>0$, depending only on $N$, such that every compact connected set $C\subset\overline{A_0}$ with $C\cap\pa B_1\ne\varnothing$ and $C\cap\pa B_2\ne\varnothing$ satisfies
\[
 \SobCap_{1,N}(C;A_1)\ge c_{\mathrm{cap}}.
\]
\end{lem}

\begin{proof}
Let $\zeta\in\mathcal C(C)$ and let $U\subset A_1$ be an open neighborhood of $C$ on which $\zeta\ge1$ almost everywhere.  Connectedness gives $C\cap\pa B_r\ne\varnothing$ for every $r\in[1,2]$.  For almost every such $r$, the restriction of $\zeta$ belongs to $W^{1,N}(\mathbb S^{N-1})$ and is at least one almost everywhere on the nonempty open set $\{\theta:r\theta\in U\}$.  Since $N>N-1$, the tangential Sobolev embedding gives
\[
 1\le C_N\int_{\mathbb S^{N-1}}
 (|\zeta(r\theta)|^N+
       |\nabla_\theta\zeta(r\theta)|^N)\,\d\theta.
\]
Integration over $1<r<2$ and the comparability of polar and Cartesian measures on this annulus yield
\[
 1\le C_N\int_{A_1}(|\zeta|^N+|\nabla\zeta|^N)\, \d y.
\]
Taking the infimum proves the claim.
\end{proof}

\begin{lem}\label{lem:bounded-sublevel-component}
Fix $0<\var <1/4$.  For all sufficiently large $R$ and every $t_R$ satisfying
\be \label{eq:tR-range}
 (1-2\var )c(R)<t_R<(1-\var )c(R),
\ee
the origin belongs to $\{v<t_R\}$, and its connected component $D_R$ satisfies
\[
 D_R\subset B_{2R}.
\]
\end{lem}

\begin{proof}
Set $h_R(y)=\ln  v(Ry)-m(R)$ and $a_\var =-\ln (1-\var )>0$.  By Proposition~\ref{prop:multiplicative-convergence}, $h_R\to0$ in $W^{1,N}(A_1)$.  Continuity shows that
\[
 \zeta_R:=\min\Big\{1,\frac{2(h_R)_-}{a_\var }\Big\}
\]
belongs to $\mathcal C(E_R^-)$ for the compact set
\[
 E_R^-:=\{y\in\overline{A_0}:h_R(y)\le-a_\var \}.
\]
Therefore
\be \label{eq:bad-capacity}
 \SobCap_{1,N}(E_R^-;A_1)
 \le C_\var \|h_R\|_{W^{1,N}(A_1)}^N\to 0.
\ee
Proposition~\ref{prop:cone-forcing} gives $c(R)\to\infty$, so $v(0)<t_R$ for large $R$.

If $D_R$ contained a point outside $B_{2R}$, its openness and connectedness would provide a path from the origin to that point in $D_R$.  The part of the path between its last visit to $\pa  B_R$ before its first visit to $\pa  B_{2R}$, rescaled by $R^{-1}$, would give a compact connected set $C_R\subset\overline{A_0}$ with $C_R\cap\pa B_1\ne\varnothing$ and $C_R\cap\pa B_2\ne\varnothing$.  For $y\in C_R$,
\[
 v(Ry)<t_R<(1-\var )c(R),
 \quad h_R(y)<-a_\var .
\]
Thus $C_R\subset E_R^-$, contradicting \eqref{eq:bad-capacity} and Lemma~\ref{lem:annular-capacity}.
\end{proof}

\begin{proof}[Proof of Theorem~\ref{thm:threshold}]
Fix $\var\in (0, 1/4)$ and choose any $t_R$ in \eqref{eq:tR-range}.  Let $D_R$ be given by Lemma~\ref{lem:bounded-sublevel-component} and set
\[
 s_R:=t_R^{1/2}.
\]
Then $s_R\to\infty$ and $s_R/t_R\to0$.  Fix $r>0$ and choose a constant $\Lambda=\Lambda(F)>0$ so that $\overline{B_{2R}}\subset\W_{\Lambda R}$ for every $R>0$.  Since $v$ is bounded on $\overline{\W_r}$, for large $R$ the connected set $\overline{\W_r}$ lies in $\{v<s_R\}$ and contains the origin.  Consequently,
\[
 \overline{\W_r}\subset D_R\subset B_{2R}\Subset\W_{\Lambda R}.
\]

Define on $\W_{\Lambda R}$
\[
 \psi_R(x):=
\left\{\begin{aligned}
  &\max\Big\{0,\dfrac{v(x)-s_R}{t_R-s_R}\Big\},&&x\in D_R,\\
 &1,&&x\notin D_R.
\end{aligned}\right.
\]
Every point of $\pa  D_R$ has value $v=t_R$.  Hence $\psi_R$ is continuous across $\pa  D_R$.  Moreover, $v$ is Lipschitz on the compact outer ball, and the same is true of $\psi_R$: if one point lies in $D_R$ and the other outside, apply the Lipschitz bound for $v$ at a boundary point between them on the joining segment.  Thus $\psi_R\in W^{1,N}(\W_{\Lambda R})$, it vanishes near $\overline{\W_r}$, and it equals one near $\pa \W_{\Lambda R}$.  This construction avoids any regularity assumption on the sublevel boundary.

Formula~\eqref{eq:wulff-capacity} and Proposition~\ref{prop:sublevel-boundary} give
\begin{align}
 \frac{N\kap _N}{[\ln (\Lambda R/r)]^{N-1}}
 &\le\int_{\W_{\Lambda R}}F(\nabla\psi_R)^N\, \d x\notag\\
 &\le\frac1{(t_R-s_R)^N}
       \int_{\{s_R<v<t_R\}}F(\nabla v)^N\, \d x\notag\\
 &=\frac1{(t_R-s_R)^N}\int_{s_R}^{t_R}G(\tau)\,\d\tau.
 \label{eq:capacity-chain}
\end{align}
The Sobolev gradient of $\psi_R$ equals $(t_R-s_R)^{-1}\nabla v$ almost everywhere on $D_R\cap\{s_R<v<t_R\}$ and vanishes elsewhere. On the complement of $D_R$ the function is identically one, and the Sobolev gradient of a function vanishes almost everywhere on each of its level sets. The same fact applies to the level sets $v=s_R$ and $v=t_R$. No measure-zero or regular-value property of $\pa  D_R$ is used in \eqref{eq:capacity-chain}.  Absolute integrability also gives
\be \label{eq:G-uniform}
 \sup_{\tau\ge s_R}|G(\tau)-M|
 \le\int_{\{v\ge s_R\}}|f|\, \d x\to 0.
\ee
We already know that $M\ge0$.  Combining \eqref{eq:capacity-chain} and \eqref{eq:G-uniform} yields
\be \label{eq:tR-upper}
 \frac{t_R-s_R}{\ln (\Lambda R/r)}
 \le\Big(\frac{M+o(1)}{N\kap _N}\Big)^{1/(N-1)},
\ee
where $M+o(1)$ can be chosen nonnegative.  If $M=0$, this contradicts $s_R/t_R\to0$, \eqref{eq:tR-range}, and the positive lower bound \eqref{eq:cone-normalization-lower}.  Thus $M>0$.

Writing $\gamma_0=(M/(N\kap _N))^{1/(N-1)}$ and using $s_R/t_R\to0$ in \eqref{eq:tR-upper}, we obtain
\[
 \limsup_{R\to\infty}\frac{c(R)}{\ln  R}
 \le\frac{\gamma_0}{1-2\var }.
\]
Letting $\var\downarrow0$ and using \eqref{eq:cone-normalization-lower}, we obtain
\be \label{eq:annular-scale-bounds}
 1+\frac\gamma N
 \le\liminf_{R\to\infty}\frac{c(R)}{\ln  R}
 \le\limsup_{R\to\infty}\frac{c(R)}{\ln  R}
 \le\gamma_0.
\ee
In particular, $M\ge N\kap _N(1+\gamma/N)^{N-1}$.
\end{proof}

At equality, \eqref{eq:annular-scale-bounds} determines the leading annular profile.

\begin{cor}
Under the hypotheses of Theorem~\ref{thm:threshold}, if $\gamma_0=1+\gamma/N$, then $c(R)/\ln  R\to1+\gamma/N$ and, for every bounded connected Lipschitz set $A\Subset\RN\setminus\{0\}$ and every $p\in [1,\infty)$,
\be \label{eq:equality-Lp-limit}
 \frac{u(R\,\cdot)}{\ln  R}
 \to -\Big(1+\frac\gamma N\Big)
 \quad\text{ in }\, L^p(A)\quad\text{ as }\, R\to\infty.
\ee
\end{cor}

\begin{proof}
At equality, both outer terms in \eqref{eq:annular-scale-bounds} equal $1+\gamma/N$.  The comparison of the annular normalizations and Proposition~\ref{prop:multiplicative-convergence} imply $v(R\cdot)/\ln  R\to1+\gamma/N$ in every finite $L^p(A)$.  Since $u=b_u-v$ and $b_u/\ln  R\to0$, this gives \eqref{eq:equality-Lp-limit}.
\end{proof}

\begin{rem}
Equality determines the leading annular profile in \eqref{eq:equality-Lp-limit}, not a bounded remainder. The term $-(2/N)\ln\ln F^0$ in \eqref{eq:equality-profile} gives finite absolute mass at the limiting growth rate, and \eqref{eq:equality-mass} proves sharpness.
\end{rem}

\begin{rem}\label{rem:measurable-sector}
The open-cone hypothesis in Theorem~\ref{thm:threshold} can be weakened. The proof of Proposition~\ref{prop:cone-forcing} only uses a fixed measurable set $E\subset A_1$ of positive measure such that $|V(Ry)|\ge cR^\gamma$ for all sufficiently large $R$ and almost every $y\in E$. Hence the open cone may be replaced by a measurable angular sector of positive surface measure. No global polynomial upper bound on $V$ is needed for this extension.
\end{rem}

\section{Annular compactness and the Pohozaev boundary limit}\label{sec:annular-compactness}

We compare $v$ with $F$-$N$-harmonic functions on annuli of increasing logarithmic width. Lemma~\ref{lem:small-mass-annuli} gives the small-mass scales, and \eqref{eq:comparison-gradient} estimates the critical gradient error under \eqref{eq:polynomial-upper}.

The external results are Ciraolo--Li's exponential comparison estimate and punctured-space harmonic rigidity theorem~\cite[Lemma~2.3 and Proposition~6.8]{CiraoloLi2024}. No pointwise expansion at infinity is assumed.

\subsection{Annular gradient criterion and boundary comparison}

We first compare the Pohozaev vector field with that of the logarithmic profile.

\begin{lem}\label{lem:pohozaev-boundary-comparison}
Let $u$ satisfy the hypotheses of Theorem~\ref{thm:threshold}, set
\[
 u_0=-\gamma_0\ln  F^0,\quad  q=\nabla(u-u_0),
\]
and let $\mathcal P_u=(x\cdot\nabla u)a_N(\nabla u)-x\Phi_N(\nabla u)$. On $\pa \W_R$,
\be \label{eq:log-profile-pohozaev}
 \mathcal P_{u_0}=\frac{N-1}{N}\frac{\gamma_0^N x}{(F^0)^N},\quad
 \int_{\pa \W_R}\mathcal P_{u_0}\cdot\nu\,\d\mathcal H^{N-1}
 =(N-1)\kap _N\gamma_0^N,
\ee
and
\be \label{eq:pohozaev-difference-N}
 |\mathcal P_u-\mathcal P_{u_0}|
 \le C R(R^{-1}+|q|)^{N-1}|q|
 \le C(R^{2-N}|q|+R|q|^N),
\ee
where $C$ depends only on $F$, $N$, and $\gamma_0$.
\end{lem}

\begin{proof}
By~\eqref{eq:polar-identities},
$a_N(\nabla u_0)=-\gamma_0^{N-1}x/(F^0)^N$ and
$x\cdot\nabla u_0=-\gamma_0$. Substitution in $\mathcal P_{u_0}$ and
\eqref{eq:minkowski} give~\eqref{eq:log-profile-pohozaev}.
The first inequality in~\eqref{eq:pohozaev-difference-N} follows from
\eqref{eq:a-difference}, $D\Phi_N=a_N$, and $|\nabla u_0|\le C/R$ on
$\pa W_R$. The second follows from
$(a+b)^{N-1}b\le C(a^{N-1}b+b^N)$ with $a=R^{-1}$.
\end{proof}

For $R>0$, define the annular gradient error at scale $R$ by
\[
 \mathfrak G_N(R):=
 \int_{\W_{2R}\setminus \W_R}
 F (\nabla u+\gamma_0\nabla\ln  F^0 )^N\, \d x.
\]
Vanishing annular gradient error gives the two boundary limits in \eqref{eq:PB-boundary}.

\begin{prop}
\label{prop:annular-gradient-error-PB}
Let \(u\) satisfy the assumptions of Theorem~\ref{thm:threshold}.
If
\[
 \liminf_{R\to\infty}\mathfrak G_N(R)=0,
\]
then \(u\) satisfies the Pohozaev condition at infinity.
\end{prop}

\begin{proof}
Choose \(R_j\to\infty\) such that \(\mathfrak G_N(R_j)\to0\), and set
\[
 w:=u+\gamma_0\ln  F^0,
 \quad q:=\nabla w.
\]
Since \(f\in L^1(\RN)\),
\[
 \int_{\W_{2R_j}\setminus \W_{R_j}}|f|\, \d x\to 0.
\]
Also, $\mathfrak G_N(R_j)\to0$ and norm equivalence give $\int_{\W_{2R_j}\setminus\W_{R_j}}|q|^N\,\d x\to0$.  Applying the coarea formula on
\(\W_{2R_j}\setminus \W_{R_j}\) to \(|f|+|q|^N\), we find
\(r_j\in[R_j,2R_j]\) such that
\be \label{eq:annular-good-radius-N}
 r_j\int_{\pa  \W_{r_j}}
  (|f|+|q|^N )\,\d\mathcal H^{N-1}\to 0.
\ee
Formula~\eqref{eq:annular-good-radius-N} gives the vanishing boundary mass in \eqref{eq:PB-boundary}.

Let \(u_0=-\gamma_0\ln  F^0\).  Estimate~\eqref{eq:pohozaev-difference-N} in Lemma~\ref{lem:pohozaev-boundary-comparison} gives
\[
 |\mathcal P_u-\mathcal P_{u_0}|
 \le C_F (r_j^{2-N}|q|+r_j|q|^N )
 \quad\hbox{ on }\,\pa  \W_{r_j}.
\]
Writing
\[
 A_j:=\int_{\pa  \W_{r_j}}|q|^N\,\d\mathcal H^{N-1},
\]
condition \eqref{eq:annular-good-radius-N} gives \(r_jA_j\to0\). H\"older's inequality and
surface scaling, \(\mathcal H^{N-1}(\pa  \W_r)\asymp r^{N-1}\), imply
\[
 r_j^{2-N}\int_{\pa  \W_{r_j}}|q|\,\d\mathcal H^{N-1}
 \le C_F(r_jA_j)^{1/N}.
\]
Hence
\[
 \int_{\pa  \W_{r_j}}
 |\mathcal P_u-\mathcal P_{u_0}|\,\d\mathcal H^{N-1}
 \le C_F(r_jA_j)^{1/N}+C_Fr_jA_j\to 0.
\]
The formula \eqref{eq:log-profile-pohozaev} for the logarithmic profile now gives \eqref{eq:PB-boundary}.
\end{proof}

By Proposition~\ref{prop:annular-gradient-error-PB}, it remains to prove $\liminf_{R\to\infty}\mathfrak G_N(R)=0$ under \eqref{eq:polynomial-upper}.

\subsection{Small-mass annuli and \texorpdfstring{$F$-$N$-harmonic}{F-N-harmonic} functions}

For \(n\ge0\), set
\be \label{eq:dyadic-f-mass}
 E_n:=\W_{2^{n+1}}\setminus \W_{2^n},
 \quad
 \mu_n:=\int_{E_n}|f|\, \d x.
\ee
Then $\sum_{n=0}^{\infty}\mu_n<\infty$. Lemma~\ref{lem:small-mass-annuli} gives annuli with sufficiently small absolute mass for \eqref{eq:comparison-gradient}.

\begin{lem}\label{lem:small-mass-annuli}
Let \((\mu_n)_{n\ge0}\in\ell^1\) be nonnegative and extend it by zero to
negative indices.  For every \(A>0\), there are integers \(n_j\to\infty\)
such that, with
\[
 L_j:=\lfloor A\ln (n_j+2)\rfloor,
 \quad
 \delta_j:=\sum_{|m-n_j|\le2L_j+4}\mu_m,
\]
one has
\be \label{eq:small-mass-annuli-rate}
 n_j\delta_j\to0\quad\text{ as }\,j\to\infty.
\ee
In particular, \(L_j=o(n_j)\) and \(n_j-2L_j\to\infty\).
\end{lem}

\begin{proof}
Let
\[
 L(n)=\lfloor A\ln (n+2)\rfloor,
 \quad
 b_n=\sum_{|m-n|\le2L(n)+4}\mu_m.
\]
If \(\liminf_{n\to\infty}nb_n>0\), then for some \(\var >0\) and all sufficiently
large \(n\), \(b_n\ge\var /n\).  Summing over \(K\le n\le2K\) gives
\[
 \sum_{n=K}^{2K}b_n\ge\var\sum_{n=K}^{2K}\frac1n\ge\frac{\var}{2}.
\]
For large \(K\), every index \(m\) counted on the left belongs to
\([K/2,3K]\), and each \(m\) occurs at most \(C_A\ln K\) times.
Consequently,
\[
 \sum_{\substack{m\in\mathbb Z\\K/2\le m\le3K}}\mu_m
 \ge \frac{c\var }{C_A\ln  K}.
\]
For \(K=8^k\), the intervals \([K/2,3K]\) are pairwise disjoint, but
\(\sum_{k=1}^{\infty}1/\ln (8^k)=\infty\).  This contradicts \(\sum_{m=0}^{\infty}\mu_m<\infty\).
Therefore \(\liminf_{n\to\infty}nb_n=0\). We may thus choose a subsequence $n_j\to\infty$ such that $n_jb_{n_j}\to0$, which gives \eqref{eq:small-mass-annuli-rate}.
\end{proof}

The exponential comparison estimate gives the following $L^p$ and critical gradient estimates, uniformly in the domain and the boundary data.

\begin{lem}\label{lem:nonlinear-comparison}
Let \(\Om \subset\RN\) be a bounded Lipschitz domain and let
\(U,h\in W^{1,N}(\Om )\) satisfy \(U-h\in W_0^{1,N}(\Om )\) and
\be \label{eq:comparison-equations-general}
 \diver a_N(\nabla U)=g,
 \quad
 \diver a_N(\nabla h)=0
 \quad\text{ in }\, \Om ,
\ee
where \(g\in L^1(\Om )\cap L^\infty(\Om )\). Set
\[
 \delta:=\|g\|_{L^1(\Om )},
 \quad \Lambda:=\|g\|_{L^\infty(\Om )}.
\]
If \(\delta=0\), then \(U=h\). If \(\delta>0\), define the dimensionless entropy
\[
 \mathcal H_\Om (g)
 :=\frac1\delta\int_\Om  |g|\ln \frac{|\Om ||g|}{\delta}\,\d x,
 \quad 0\ln 0:=0.
\]
It satisfies $0\le\mathcal H_\Om (g)\le\ln (\Lambda|\Om |/\delta)$. For every finite \(p\ge1\),
\be \label{eq:comparison-Lp}
 \|U-h\|_{L^p(\Om )}
 \le C_{p,F,N}|\Om |^{1/p}\delta^{1/(N-1)},
\ee
and
\begin{align}\label{eq:comparison-gradient}
 \int_\Om  F(\nabla U-\nabla h)^N\,\d x
 &\le C_{F,N}\delta^{N/(N-1)}(1+\mathcal H_\Om (g))\notag\\
 &\le C_{F,N}\delta^{N/(N-1)}
 \Big[1+\ln \Big(e+\frac{\Lambda|\Om |}{\delta}\Big)\Big].
\end{align}
\end{lem}

\begin{proof}
Set \(z=U-h\).  Applying the anisotropic Brezis--Merle estimate
of Ciraolo--Li~\cite[Lemma~2.3]{CiraoloLi2024} to \(-g\), let $\mu_{F,N}>0$ denote the constant in that estimate and choose
\[
 \lam =\frac{\mu_{F,N}}2\,\delta^{-1/(N-1)}.
\]
Then
\[
 \int_\Om  e^{\lam |z|}\, \d x\le2|\Om |.
\]
For each finite $p\ge1$, the inequality $s^p\le C_pe^s$ for $s\ge0$ gives
$\int_\Om |z|^p\,\d x\le C_p|\Om |\lam ^{-p}$, proving \eqref{eq:comparison-Lp}.

We next estimate the integral involving $g$. Set $t=|\Om ||g|/\delta$. Jensen's inequality for $t\ln t$ gives $\mathcal H_\Om(g)\ge0$, and $t\le\Lambda|\Om|/\delta$ gives the upper bound. The elementary inequality
\[
 t b\le t\ln  t-t+e^b\quad(t\ge0,\ b\in\R)
\]
with $b=\lam |z|$, integrated against $\d x/|\Om |$, yields
\[
 \frac{\lam }{\delta}\int_\Om  |g||z|\,\d x
 \le\mathcal H_\Om (g)-1+\frac1{|\Om |}\int_\Om  e^{\lam |z|}\,\d x
 \le\mathcal H_\Om (g)+1.
\]
Consequently,
\be \label{eq:gz-bound}
 \int_\Om  |g||z|\,\d x
 \le C_{F,N}\delta^{N/(N-1)}(1+\mathcal H_\Om (g)).
\ee
Finally, subtract the two weak equations in
\eqref{eq:comparison-equations-general} and test with \(z\).  The strong
monotonicity estimate \eqref{eq:a-monotonicity}, norm equivalence, and
\eqref{eq:gz-bound} give
\[
 c_{F,N}\int_\Om  F(\nabla z)^N\, \d x
 \le \Big|\int_\Om  gz\, \d x \Big|
 \le\int_\Om |g||z|\, \d x,
\]
which is \eqref{eq:comparison-gradient}.
\end{proof}

\begin{rem}
The proof combines a distribution estimate for $U-h$ with
strong monotonicity.  Thus its constants remain independent of the
possibly vanishing gradients of the two solutions.  This is the form needed
when $N>2$.
\end{rem}

\subsection{Proof of Theorem~\ref{thm:annular-pohozaev}}

For \(k\in\mathbb Z\), define the Wulff annuli $\mathcal B_k\Subset\mathcal A_k$ by
\[
 \mathcal A_k:=\{2^{k-2}<F^0<2^{k+2}\},
 \quad
 \mathcal B_k:=\{2^{k-1}<F^0<2^{k+1}\},
\]
and set
\[
 \om _k(h):=\operatorname*{osc}_{\mathcal A_k}h
 :=\sup_{\mathcal A_k}h-\inf_{\mathcal A_k}h.
\]
We call $h$ $F$-$N$-harmonic when it is a weak solution of $\diver a_N(\nabla h)=0$. If $h$ is $F$-$N$-harmonic on $W_{r_2}\setminus\overline{W_{r_1}}$, the boundary quantity
\[
 \mathcal Q(h):=
 \int_{\pa  \W_r}a_N(\nabla h)\cdot\nu\, \d\mathcal H^{N-1}
\]
is independent of $r\in(r_1,r_2)$. We compare $\mathcal Q(h)$ with the oscillation of $h$, starting with Caccioppoli's inequality.

\begin{lem}
\label{lem:harmonic-caccioppoli}
If \(h\) is \(F\)-\(N\)-harmonic on \(\mathcal A_k\), then
\be \label{eq:harmonic-caccioppoli}
 \int_{\mathcal B_k}F(\nabla h)^N\, \d x
 \le C_{F,N}\,\om _k(h)^N.
\ee
\end{lem}

\begin{proof}
Choose $\eta\in C_c^1(\mathcal A_k)$ with $\eta=1$ on $\mathcal B_k$ and $|\nabla\eta|\le C2^{-k}$, and choose $c$ between the infimum and supremum of $h$ on $\mathcal A_k$. Testing with $\eta^N(h-c)$, using \eqref{eq:a-growth}, and applying Young's inequality give
\[
 \int_{\mathcal B_k}F(\nabla h)^N\,\d x
 \le C_{F,N}\int_{\mathcal A_k}|h-c|^N|\nabla\eta|^N\,\d x
 \le C_{F,N}\om _k(h)^N,
\]
since $|\mathcal A_k|\asymp2^{Nk}$. This proves \eqref{eq:harmonic-caccioppoli}.
\end{proof}

The homogeneity of $a_N$ gives $\mathcal Q(ch)=c^{N-1}\mathcal Q(h)$ for $c>0$. We shall also use the following result of Ciraolo and Li~\cite[Proposition~6.8]{CiraoloLi2024}: an $F$-$N$-harmonic function $H$ on $\RN\setminus\{0\}$ satisfying
\[
 |H(x)|\le C(1+|\ln  F^0(x)|)
\]
has the form $H=\beta\ln  F^0+c$. Norm equivalence gives their Euclidean logarithmic growth bound, and $F^0(-x)=F^0(x)$ identifies their two dual norms. The growth bound is sufficient, without a sign condition on $H$ or a bounded remainder relative to a specified logarithm.

\begin{lem}\label{lem:annular-growth}
There are an integer \(J_F\ge2\) and \(\var _F>0\), depending only on
\(F\) and \(N\), with the following property.  Suppose that \(h\) is
\(F\)-\(N\)-harmonic on an annulus containing \(\mathcal A_k\) for all
\(|k|\le J_F\).  If \(\om _0(h)>0\) and
\[
 |\mathcal Q(h)|\le\var _F\om _0(h)^{N-1},
\]
then
\[
 \max_{|k|\le J_F}\om _k(h)\ge2\om _0(h).
\]
The same conclusion holds after shifting the indices.
\end{lem}

\begin{proof}
If the statement were false, then for every \(j\) one could find an
\(F\)-\(N\)-harmonic function \(h_j\) on an annulus covering the dyadic
indices \(|k|\le j\) such that, after subtracting a constant and dividing by
its central oscillation,
\[
 \om _0(h_j)=1,
 \quad
 |\mathcal Q(h_j)|\le j^{-1},
 \quad
 \max_{|k|\le j}\om _k(h_j)\le2.
\]
Fix \(y_*\in\pa  \W_1\) and normalize \(h_j(y_*)=0\).  The intersections $\mathcal A_k\cap\mathcal A_{k+1}\ne\varnothing$ and the oscillation bounds give
\[
 |h_j(x)|\le2(1+|k|),\quad x\in\mathcal A_k,\quad |k|\le j.
\]
Lemma~\ref{lem:local-C1-compactness}, applied with zero right-hand side on a covering by fixed interior balls, and a diagonal argument give, after
passing to a subsequence,
\[
 h_j\to  h_\infty
 \quad\text{ in }\, C^1_{loc}(\RN\setminus\{0\}).
\]
The limit is $F$-$N$-harmonic, with $\om_0(h_\infty)=1$ and $\mathcal Q(h_\infty)=0$, and obeys
\[
 |h_\infty(x)|\le C(1+|\ln  F^0(x)|).
\]
Ciraolo--Li~\cite[Proposition~6.8]{CiraoloLi2024} and $F^0(-x)=F^0(x)$ give
\[
 h_\infty=\beta\ln  F^0+c
\]
for some \(\beta,c\in\R\).  By \eqref{eq:log-boundary}, its boundary integral is
\[
 \mathcal Q(h_\infty)
 =N\kap _N|\beta|^{N-2}\beta.
\]
The zero boundary integral gives \(\beta=0\), contradicting
\(\om _0(h_\infty)=1\).
\end{proof}

We next compare the boundary integrals of the solution and its $F$-$N$-harmonic comparison function. Estimate~\eqref{eq:a-difference} gives the comparison directly, without a linearized equation.

\begin{lem}\label{lem:boundary-stability}
Let $D$ be an annulus containing $\overline{\mathcal A_k}$ and let
$U,h\in W^{1,N}_{loc}(D)$, with $h$ a weak solution of $\diver a_N(\nabla h)=0$ in $D$.
Set $z=U-h$ and $I_k=[2^{k-1/2},2^{k+1/2}]$. For any prescribed measurable set $E_k\subset I_k$ with $|I_k\setminus E_k|=0$, there exists $r_k\in E_k$ such that
\begin{align}
 & \,\Big|
 \mathcal Q(h)-
 \int_{\pa  \W_{r_k}}a_N(\nabla U)\cdot\nu\,\d\mathcal H^{N-1}
  \Big|\notag\\
 \le&\,
 C_{F,N}\om _k(h)^{N-2}
 \Big(\int_{\mathcal B_k}|\nabla z|^N\, \d x\Big)^{1/N}
 +C_{F,N}
 \Big(\int_{\mathcal B_k}|\nabla z|^N\, \d x\Big)^{(N-1)/N}.
 \label{eq:boundary-stability}
\end{align}
\end{lem}

\begin{proof}
All surface integrals involving $U$ below are taken at radii where their
coarea representatives are defined.  By \eqref{eq:a-difference},
\[
 |a_N(\nabla U)-a_N(\nabla h)|
 \le C_{F,N}
  (|\nabla h|^{N-2}|\nabla z|+|\nabla z|^{N-1} ).
\]
Let
\[
 I_k=[2^{k-1/2},2^{k+1/2}],
 \quad
 \mathcal S_k=\{x:F^0(x)\in I_k\}\Subset\mathcal B_k.
\]
The coarea formula and the uniform bounds for \(|\nabla F^0|\) give
\begin{align*}
 \int_{I_k}\int_{\pa  \W_r}
 |a_N(\nabla U)-a_N(\nabla h)|\,\d\mathcal H^{N-1}\, \d r\le C_F\int_{\mathcal S_k}
 |a_N(\nabla U)-a_N(\nabla h)|\, \d x.
\end{align*}
Coarea defines the surface integrals for almost every $r\in I_k$. Since $E_k$ has full measure and $|I_k|\asymp2^k$, averaging gives $r_k\in E_k$ where the boundary integral is at most $C2^{-k}$ times the volume integral. For a positive average use twice its value; for a zero average the integral vanishes almost everywhere.

Write
\[
 Z_k:=\int_{\mathcal B_k}|\nabla z|^N\, \d x.
\]
Using H\"older's inequality with a third factor equal to one,
Lemma~\ref{lem:harmonic-caccioppoli}, and
\(|\mathcal S_k|^{1/N}\asymp2^k\), we obtain
\begin{align*}
 \int_{\mathcal S_k}|\nabla h|^{N-2}|\nabla z|\, \d x
 \le
 \Big(\int_{\mathcal B_k}|\nabla h|^N\,\d x\Big)^{(N-2)/N}
 Z_k^{1/N}|\mathcal S_k|^{1/N}\le C_{F,N}\om _k(h)^{N-2}Z_k^{1/N}2^k,
\end{align*}
and similarly
\[
 \int_{\mathcal S_k}|\nabla z|^{N-1}\, \d x
 \le C_{F,N}Z_k^{(N-1)/N}2^k.
\]
For $N=2$, the first term in \eqref{eq:boundary-stability} means $CZ_k^{1/2}$,
including when $\om _k(h)=0$.  The factor \(2^k\) cancels the averaging
factor \(2^{-k}\).  Since the boundary integral
of \(h\) is independent of the radius, this proves
\eqref{eq:boundary-stability}.
\end{proof}

\begin{proof}[Proof of Theorem~\ref{thm:annular-pohozaev}]
Recall
\[
 v=b_u-u\ge1,
 \quad
 \diver a_N(\nabla v)=f,
 \quad
 f=Ve^{Nu}\in L^1(\RN).
\]
By Theorem~\ref{thm:threshold}, \(M>0\) and
\[
 \gamma_0=\Big(\frac{M}{N\kap _N}\Big)^{1/(N-1)}>0.
\]
We first obtain compactness for the functions $h_j$, then transfer it to the
solution.  All constants below are independent of $j$ and of the dyadic index in the indicated range. Unless specified otherwise, limits involving $j$ are taken as $j\to\infty$.

\textbf{Step 1:} uniform multiplicative estimates and small-mass scales.

Fix the Wulff annulus
\[
 \mathcal A_*:=\{2^{-3}<F^0<2^3\}
\]
and write
\[
 c_*(S):=\exp\Big(\dashint_{\mathcal A_*} \ln  v(Sy)\, \d y\Big).
\]
The proof of Proposition~\ref{prop:multiplicative-convergence}, together with the
fact that the exterior part of the integral in \eqref{eq:automatic-log-energy} tends to zero, gives for every finite
\(p\)
\be \label{eq:uniform-convergence-N}
 \sup_{S\ge T}
 \Big\|\frac{v(S\,\cdot)}{c_*(S)}-1 \Big\|_{L^p(\mathcal A_*)}
 \to 0
 \quad\text{ as }\, T\to\infty.
\ee
Indeed, the rescaled critical energy is bounded by the exterior integral
\(\int_{\{F^0>cT\}}F(\nabla\ln v)^N\,\d x\), which tends to zero uniformly in
\(S\ge T\).  Poincar\'e and the critical Moser--Trudinger inequality then give
\eqref{eq:uniform-convergence-N} exactly as before.

Estimate~\eqref{eq:annular-scale-bounds} and Remark~\ref{rem:annular-normalizations} give
\be \label{eq:uniform-c-upper-N}
 c_*(S)\le C\ln  S
 \quad\hbox{for all sufficiently large }S.
\ee

Choose \(A>0\) so large that
\be \label{eq:A-choice-N}
 \frac{A\ln 2}{2J_F}>2,
\ee
where \(J_F\) is from Lemma~\ref{lem:annular-growth}.  Apply
Lemma~\ref{lem:small-mass-annuli} to the sequence in
\eqref{eq:dyadic-f-mass}.  With the resulting \(n_j,L_j,\delta_j\), set
\[
 R_j:=2^{n_j},
 \quad
 \Om _j:=\{2^{-2L_j}<F^0<2^{2L_j}\},
\]
\[
 v_j(y):=v(R_jy),
 \quad
 f_j(y):=R_j^Nf(R_jy).
\]
Then
\be \label{eq:Fj-L1-N}
 d_j:=\|f_j\|_{L^1(\Om _j)}\le\delta_j.
\ee
Moreover, the independent polynomial upper bound \eqref{eq:polynomial-upper}
and \(u\le\sup u\) give
\[
 \ln  (e+\|f_j\|_{L^\infty(\Om _j)} )
 \le C(n_j+L_j),
\]
and
\be \label{eq:Omega-volume-N}
 |\Om _j|\le C_F2^{2NL_j}.
\ee

\textbf{Step 2:} $F$-$N$-harmonic functions and uniform bounds.

Let \(h_j\) solve the \(F\)-\(N\)-harmonic Dirichlet problem on
\(\Om _j\) with boundary data \(v_j\):
\[
 \begin{cases}
  \diver a_N(\nabla h_j)=0&\text{ in }\, \Om _j,\\
  h_j=v_j&\text{ on }\, \pa \Om _j.
 \end{cases}
\]
Minimization of $\int_{\Om _j}\Phi_N(\nabla h)\,\d y$ on $v_j+W^{1,N}_0(\Om _j)$ gives a unique $h_j$ by coercivity and strict convexity. The boundary values are understood in the trace sense. Since \(v_j\ge1\) on the boundary, comparison with the constant solution
implies \(h_j\ge1\) in \(\Om _j\).  Set \(z_j=v_j-h_j\).  If \(d_j=0\),
uniqueness for the homogeneous Dirichlet problem gives \(z_j\equiv0\).  For
\(d_j>0\), Lemma~\ref{lem:nonlinear-comparison} and
\eqref{eq:Fj-L1-N}--\eqref{eq:Omega-volume-N} give
\be \label{eq:zj-energy-N}
 \mathcal Z_j:=\int_{\Om _j}F(\nabla z_j)^N\, \d y
 \le Cd_j^{N/(N-1)}
  (1+n_j+L_j+|\ln  d_j| )
 \to 0.
\ee
We use the convention \(\mathcal Z_j=0\) when \(d_j=0\).  The convergence
follows from \(n_jd_j\le n_j\delta_j\to0\), \(L_j=o(n_j)\), and
\(d_j^{N/(N-1)}|\ln  d_j|\to0\).

We next estimate $h_j$ using $h_j=v_j-z_j$.
Choose once and for all a finite exponent \(p\) so large that
\[
 \frac{3NA\ln 2}{p}<\frac1{N-1}.
\]
For large $j$, all annuli $\{2^{k-3}<F^0<2^{k+3}\}$ with $|k|\le L_j$ are contained in $\Om _j$, since $L_j\to\infty$. For these indices, scaling
\eqref{eq:uniform-convergence-N} to the physical radius \(R_j2^k\), together
with \eqref{eq:uniform-c-upper-N} and \(L_j=o(n_j)\), gives
\[
 \Big(\dashint_{\{2^{k-3}<F^0<2^{k+3}\}} |v_j|^p\, \d y\Big)^{1/p}
 \le Cn_j.
\]
On the other hand, \eqref{eq:comparison-Lp} and
\eqref{eq:Omega-volume-N} imply, uniformly for \(|k|\le L_j\),
\begin{align*}
 \Big(\dashint_{\{2^{k-3}<F^0<2^{k+3}\}} |z_j|^p\, \d y\Big)^{1/p}
 \le C_p d_j^{1/(N-1)}2^{3NL_j/p}\to 0.
\end{align*}
The volume ratio is largest when $k=-L_j$. In fact,
if \(\theta:=3NA\ln 2/p<1/(N-1)\), then
\[
 d_j^{1/(N-1)}2^{3NL_j/p}
 \le C(n_jd_j)^{1/(N-1)}
 n_j^{\theta-1/(N-1)}\to 0.
\]
Thus the $L^p$ averages of $h_j$ are $O(n_j)$. Rescale
\(\{2^{k-3}<F^0<2^{k+3}\}\) to the fixed annulus
\(\mathcal A_*\).  The positive function \(h_j\) is \(F\)-\(N\)-harmonic there.  Serrin's local boundedness estimate~\cite[Theorem~2]{Serrin1964}, applied on finitely many interior balls of the fixed annulus, gives
\[
 1\le h_j\le Cn_j,
 \quad\hbox{ on }\,\mathcal A_k,
 \quad
 \om _k(h_j)\le Cn_j
 \quad \forall\, k\in\mathbb Z\ \text{with }|k|\le L_j-4.
\]

\textbf{Step 3:} uniform oscillation estimates and identification of the boundary integral.

Let
\[
 \mathcal Q_j:=\mathcal Q(h_j).
\]
For \(|k|\le L_j-4\), Lemma~\ref{lem:boundary-stability}, norm equivalence, and
\eqref{eq:zj-energy-N} give a radius
\(r_{j,k}\asymp2^k\) such that
\[
 \Big|
 \mathcal Q_j-
 \int_{\pa  \W_{r_{j,k}}}a_N(\nabla v_j)\cdot\nu\,
 \d\mathcal H^{N-1}
  \Big|
 \le C\om _k(h_j)^{N-2}\mathcal Z_j^{1/N}
 +C\mathcal Z_j^{(N-1)/N}.
\]
The radius can also be chosen in the full-measure set on which the weak divergence identity for the global solution holds. Thus
\[
 \int_{\pa  \W_{r_{j,k}}}a_N(\nabla v_j)\cdot\nu\,\d\mathcal H^{N-1}
 =\int_{\W_{r_{j,k}}}f_j\, \d y
 =\int_{\W_{R_jr_{j,k}}}f\, \d x.
\]
Set
\[
 \tau_j:=
 \int_{\{F^0>cR_j2^{-L_j}\}}|f|\, \d x,
\]
where \(c>0\) is fixed and small.  Since
\(R_j2^{-L_j}=2^{n_j-o(n_j)}\to\infty\), one has \(\tau_j\to0\), and the
integrals of $f$ above differ from \(M\) by at most \(\tau_j\), uniformly
for \(|k|\le L_j-4\).  Consequently
\be \label{eq:boundary-master-N}
 |\mathcal Q_j-M|
 \le \tau_j
 +C\om _k(h_j)^{N-2}\mathcal Z_j^{1/N}
 +C\mathcal Z_j^{(N-1)/N},
 \quad \forall\, k\in\mathbb Z\ \text{with }|k|\le L_j-4.
\ee

We now prove the scale-independent bound
\be \label{eq:omega-uniform-N}
 \sup_{\substack{k\in\mathbb Z\\|k|\le L_j/2}}\om_k(h_j)\le K
\ee
for a fixed constant $K$. For large $j$ we have $\tau_j\le1$ and $\mathcal Z_j\le1$. Choose $K\ge1$, independently of $j$, so large that
\[
 \frac{M+1}{K^{N-1}}+\frac C K+\frac C{K^{N-1}}<\var _F,
\]
with the constants from \eqref{eq:boundary-master-N}. Suppose an index $k_j$ with $|k_j|\le L_j/2$ satisfies $\om _{k_j}(h_j)\ge K$. Set
\[
 M_m:=\max_{|\ell-k_j|\le mJ_F}\om _\ell(h_j),\quad
 m_j:=\Big\lfloor\frac{L_j/2-8}{J_F}\Big\rfloor.
\]
For $j$ large, $m_j\ge1$. Whenever $0\le m<m_j$, choose a maximizing index $\ell_m$. All indices up to the next step have absolute value at most $L_j-8$, so both the boundary estimate and the uniform $Cn_j$ bound remain available. Dividing \eqref{eq:boundary-master-N} at $\ell_m$ by $M_m^{N-1}$ gives
\be \label{eq:relative-boundary-N}
 \frac{|\mathcal Q_j|}{M_m^{N-1}}
 \le\frac{M+\tau_j}{M_m^{N-1}}
   +C\frac{\mathcal Z_j^{1/N}}{M_m}
   +C\frac{\mathcal Z_j^{(N-1)/N}}{M_m^{N-1}}
 <\var _F.
\ee
Here $M_m\ge K$. Lemma~\ref{lem:annular-growth}, centered at $\ell_m$, yields $M_{m+1}\ge2M_m$. Iteration therefore gives
\[
 Cn_j\ge M_{m_j}\ge 2^{m_j}K.
\]
But $m_j\ge L_j/(2J_F)-C_F$ and $L_j=\lfloor A\ln (n_j+2)\rfloor$, so
\[
 2^{m_j}\ge c_F(n_j+2)^{A\ln 2/(2J_F)}.
\]
This contradicts \eqref{eq:A-choice-N}. Thus \eqref{eq:omega-uniform-N} holds. The division in \eqref{eq:relative-boundary-N} precedes the limit: for $N>2$, the preliminary $O(n_j)$ oscillation bound together with $\mathcal Z_j\to0$ would not by itself make $\om _k(h_j)^{N-2}\mathcal Z_j^{1/N}$ tend to zero.

Now take \(k=0\) in \eqref{eq:boundary-master-N}.  The uniform oscillation bound
and \(\mathcal Z_j,\tau_j\to0\) give
\be \label{eq:Qj-to-M-N}
 \mathcal Q_j\to  M>0.
\ee

\textbf{Step 4:} limiting logarithmic profile and vanishing annular gradient error.

Fix \(y_*\in\pa  \W_1\).  The intersections $\mathcal A_k\cap\mathcal A_{k+1}\ne\varnothing$ and
\eqref{eq:omega-uniform-N} give
\[
 |h_j(y)-h_j(y_*)|
 \le C(1+|k|),
 \quad \forall\, k\in\mathbb Z\ \text{with }|k|\le L_j/2-2,\ \forall\, y\in\mathcal A_k.
\]
Lemma~\ref{lem:local-C1-compactness} and a diagonal argument therefore give,
after passing to a subsequence,
\[
 h_j-h_j(y_*)\to  h_\infty
 \quad\text{ in }\, C^1_{loc}(\RN\setminus\{0\}).
\]
The limit is \(F\)-\(N\)-harmonic and satisfies
\[
 |h_\infty(y)|\le C(1+|\ln  F^0(y)|).
\]
By Ciraolo--Li's rigidity theorem on $\RN\setminus\{0\}$
\cite[Proposition~6.8]{CiraoloLi2024} and $F^0(-x)=F^0(x)$,
\[
 h_\infty(y)=\beta\ln  F^0(y)+c
\]
for some \(\beta,c\in\R\).  Our normalization at \(y_*\), where
\(F^0(y_*)=1\), gives \(c=0\).  Moreover, local \(C^1\) convergence on a
fixed Wulff sphere and \eqref{eq:Qj-to-M-N} give
\[
 M=\mathcal Q(h_\infty)
 =N\kap _N|\beta|^{N-2}\beta.
\]
Since \(M>0\), \(\beta>0\), and by the definition of \(\gamma_0\),
\(\beta=\gamma_0\).  Every convergent subsequence has the same limit, hence
\be \label{eq:hj-log-convergence-N}
 h_j-h_j(y_*)\to \gamma_0\ln  F^0
 \quad\text{ in }\, C^1_{loc}(\RN\setminus\{0\}).
\ee

Let $0<a<b<\infty$ be fixed.  By \eqref{eq:zj-energy-N} and
\eqref{eq:hj-log-convergence-N},
\be \label{eq:all-annular-gradient-limit}
 \int_{\W_b\setminus\W_a}
 F (\nabla v_j-\gamma_0\nabla\ln  F^0 )^N\, \d y
 \to 0.
\ee
Convergence~\eqref{eq:hj-log-convergence-N} holds on every compact subset of $\RN\setminus\{0\}$, so one sequence suffices for all $0<a<b<\infty$. The constants $h_j(y_*)$ need not converge. The evenness of $F$ and critical scale invariance give
\be \label{eq:all-annular-gradient-error-limit}
 \int_{\W_{bR_j}\setminus\W_{aR_j}}
 F (\nabla u+\gamma_0\nabla\ln  F^0 )^N\, \d x
 \to 0.
\ee
Since $a$ and $b$ are arbitrary, \eqref{eq:all-annular-gradient-error-limit} proves
the first limit in \eqref{eq:auto-annular-gradient-error}; the energy limit is computed below. Taking $(a,b)=(1,2)$ and applying the coarea selection in Proposition~\ref{prop:annular-gradient-error-PB} gives radii $r_j\in[R_j,2R_j]$ satisfying both limits in \eqref{eq:PB-boundary}. Passing to a subsequence, we may assume that $r_j$ is strictly increasing.

Take $\mathcal A$ and $\alpha_j$ as in the theorem, and set
$U_j(y)=u(R_jy)+\gamma_0\ln F^0(y)-\alpha_j$.
For any fixed Wulff annulus $D\Subset\RN\setminus\{0\}$ containing
$\overline{\mathcal A}$, \eqref{eq:all-annular-gradient-error-limit} gives
$\|\nabla U_j\|_{L^N(D)}\to0$. The normalization
$\dashint_{\mathcal A}U_j\,\d y=0$ and Poincar\'e's inequality give
$\|U_j\|_{W^{1,N}(D)}\to0$. Every compact subset of the punctured space
is contained in such an annulus, proving the stated strong convergence.
The Sobolev embedding also gives convergence in every finite local
$L^p$ space.

Finally, $|F(\xi)-F(\eta)|\le F(\xi-\eta)\le C_F|\xi-\eta|$,
so~\eqref{eq:all-annular-gradient-limit} gives convergence of the corresponding
$L^N$ gradient norms.  The Wulff polar formula and
$F(\nabla\ln  F^0)=1/F^0$ yield
\begin{align*}
 \lim_{j\to\infty}
 \int_{\W_{bR_j}\setminus\W_{aR_j}}F(\nabla u)^N\, \d x
=\gamma_0^N\int_{\W_b\setminus\W_a}\frac{\d y}{F^0(y)^N}=N\kap _N\gamma_0^N\ln \frac ba.
\end{align*}
This proves all the compactness and energy conclusions stated in the theorem.
\end{proof}

\begin{prop}\label{prop:larger-upper-growth}
In Theorem~\ref{thm:annular-pohozaev}, the polynomial bound \eqref{eq:polynomial-upper} may be replaced by \eqref{eq:larger-upper-growth}.
\end{prop}

\begin{proof}
Use the same $n_j,L_j,\delta_j,\Om _j$ and $d_j$ as above, and set $q_N=N/(N-1)$. The upper bound on $u$ and \eqref{eq:larger-upper-growth} give
\[
 \ln (e+\|f_j\|_{L^\infty(\Om _j)})
 \le C(1+(n_j+L_j)^{q_N}).
\]
Indeed, $f_j(y)=2^{Nn_j}V(2^{n_j}y)e^{Nu(2^{n_j}y)}$ and $|y|\le C_F2^{2L_j}$ in $\Om _j$. The gradient-difference term is therefore bounded by
\[
 C d_j^{q_N}(1+(n_j+L_j)^{q_N}+L_j+|\ln  d_j|)\to0,
\]
with the case $f_j\equiv0$ treated as before. The convergence follows from $L_j=o(n_j)$, \eqref{eq:Fj-L1-N}, and \eqref{eq:small-mass-annuli-rate}: the leading term is at most $C(n_jd_j)^{q_N}\to0$, and $d_j^{q_N}|\ln d_j|\to0$.

The $L^p$ bound for $v_j-h_j$, the choice of $p$, and every subsequent oscillation and boundary estimate are unchanged: none uses the polynomial upper bound except through the logarithm of $\|f_j\|_\infty$. The conclusions of Theorem~\ref{thm:annular-pohozaev} follow.
\end{proof}

\section{Pohozaev identities and rigidity}\label{sec:pohozaev}

We combine \eqref{eq:PB-boundary} with the local Pohozaev identity. The radial and directional identities below correspond to the infinitesimal dilations $x\mapsto(1+t)x$ and translations $x\mapsto x+te$, where $e\in\RN$.

\subsection{Local and global Pohozaev identities}

For $y\in\RN$, write
\[
 \mathcal P_{u,y}:=((x-y)\cdot\nabla u)a_N(\nabla u)
 -(x-y)\Phi_N(\nabla u),
 \quad \mathcal P_u=\mathcal P_{u,0}.
\]
We first prove a local Pohozaev identity for $C^1$ weak solutions. Here $\mathcal D'(\RN)$ denotes the space of distributions. No second derivatives of $u$ are used at critical points.

\begin{lem}\label{lem:local-pohozaev}
Let $u$ be a weak solution of \eqref{maineq} with $u\in C^1_{loc}(\RN)$. Then
\[
 \diver\Big(\mathcal P_u+\frac{x}{N}Ve^{Nu}\Big)
 =\Big(V+\frac{x\cdot\nabla V}{N}\Big)e^{Nu}
 \quad\text{ in }\, \mathcal D'(\RN).
\]
In particular, for every $R>0$,
\be \label{eq:pohozaev-R}
 \int_{\pa \W_R}\Big[\mathcal P_u\cdot\nu
 +\frac{x\cdot\nu}{N}Ve^{Nu}\Big]\,\d\mathcal H^{N-1}
 =\int_{\W_R}\Big(V+\frac{x\cdot\nabla V}{N}\Big)e^{Nu}\, \d x.
\ee
The same statements hold with $x$ replaced by $x-y$ in the Pohozaev vector field, the boundary term involving $Ve^{Nu}$, and the derivative term.
\end{lem}

\begin{proof}
We give a direct weak proof of the identity for variable $V$; compare the general Pohozaev identity in Ciraolo and Li~\cite[Lemma~4.2]{CiraoloLi2024}. Set $f(x)=V(x)e^{Nu(x)}$. On a bounded smooth domain $D\Subset\RN$, convexity of $\Phi_N$ and the weak equation imply
\[
 \int_D[\Phi_N(\nabla w)-fw]\,\d x
 \ge \int_D[\Phi_N(\nabla u)-fu]\,\d x,
 \quad w-u\in W^{1,N}_0(D).
\]
Here $f$ is kept fixed. Let $X\in C_c^1(D;\RN)$, set $T_t(x)=x+tX(x)$, and take $w_t=u\circ T_t^{-1}$ for $|t|$ small. Differentiation at $t=0$, using a change of variables for the gradient term, gives
\[
 \int_D[a_N(\nabla u)\cdot (DX)^T\nabla u
                 -\Phi_N(\nabla u)\diver X]\,\d x
 =\int_D f(x)X(x)\cdot\nabla u(x)\,\d x.
\]
The differentiation is justified by $u\in C^1(\overline D)$ and $\Phi_N\in C^1(\RN)$; it does not require a Hessian of $u$ or of $\Phi_N$ at zero. For $X(x)=\eta(x)(x-y)$, the identity $a_N(\nabla u)\cdot\nabla u=N\Phi_N(\nabla u)$ cancels the terms containing $\eta$ but not $\nabla\eta$. Consequently,
\[
 \diver\mathcal P_{u,y}
 =-f(x)\,(x-y)\cdot\nabla u(x)
 \quad\text{ in }\, \mathcal D'(\RN).
\]
Since $f\in C^1_{loc}$, the product rule gives
\[
 \diver\Big(\frac{x-y}{N}f(x)\Big)
 =f(x)+\frac{(x-y)\cdot\nabla V(x)}{N}e^{Nu(x)}
       +f(x)\,(x-y)\cdot\nabla u(x).
\]
Adding these identities proves the distributional formula. Testing with radial cutoffs in $F^0(x)$, by approximation, yields the boundary identity for almost every radius. Both sides are continuous in $R>0$, by continuity of the vector field, the integrand, and the parametrization $\pa  W_R=R\pa  W_1$. Thus it holds for every $R>0$. The translated vector field is treated in the same way.
\end{proof}

Lemma~\ref{lem:good-radii} gives the vanishing boundary mass in \eqref{eq:PB-boundary}. The global Pohozaev identity uses both limits in \eqref{eq:PB-boundary} for the same radii.

\begin{prop}\label{prop:global-pohozaev}
Let $u$ satisfy the hypotheses of Theorem~\ref{thm:threshold} and the Pohozaev condition at infinity (Definition~\ref{def:pohozaev-infinity}). Set $D(x)=x\cdot\nabla V(x)$. If at least one of $D_+e^{Nu}$ and $D_-e^{Nu}$ belongs to $L^1(\RN)$, then $|D|e^{Nu}\in L^1(\RN)$ and
\be \label{eq:global-pohozaev}
 \gamma_0^{N-1}((N-1)\gamma_0-N)
 =\frac1{N\kap _N}\int_{\RN}x\cdot\nabla V\,e^{Nu}\,\d x.
\ee
In particular, if $x\cdot\nabla V\le0$, then
\[
 (-x\cdot\nabla V)e^{Nu}\in L^1(\RN).
\]
\end{prop}

\begin{proof}
Choose the increasing sequence in Definition~\ref{def:pohozaev-infinity}. Since $|x\cdot\nu|\le C_Fr_j$ on $\pa \W_{r_j}$, the boundary term involving $Ve^{Nu}$ in \eqref{eq:pohozaev-R} tends to zero as $j\to\infty$. Absolute mass integrability and the prescribed Pohozaev boundary limit give
\be \label{eq:pre-global}
 (N-1)\kap _N\gamma_0^N
 =M+\frac1N\lim_{j\to\infty}\int_{\W_{r_j}}D e^{Nu}\,\d x.
\ee
Suppose $D_+e^{Nu}\in L^1$. Its integral over $\W_{r_j}$ has a finite limit. Identity \eqref{eq:pre-global} then bounds the integrals of $D_-e^{Nu}$ on these increasing balls. Monotone convergence proves that $D_-e^{Nu}\in L^1$. Interchanging the two signs proves the other case. Thus the limit in \eqref{eq:pre-global} is an absolutely convergent Lebesgue integral. Substitution of $M=N\kap _N\gamma_0^{N-1}$ gives \eqref{eq:global-pohozaev}.
\end{proof}

A polynomial upper bound on $|V|$ is not an estimate on $|\nabla V|$. Proposition~\ref{prop:global-pohozaev} assumes one-sided derivative integrability when $x\cdot\nabla V$ changes sign; the upper growth assumption alone does not imply it.

\subsection{Boundary criteria and translation identities}

Finite exterior gradient error gives both limits in \eqref{eq:PB-boundary} for the same radii.
\begin{prop}\label{prop:finite-gradient-error}
A solution satisfying the hypotheses of Theorem~\ref{thm:threshold} and $\cG_\infty^{(N)}(u)<\infty$ satisfies the Pohozaev condition at infinity.
\end{prop}

\begin{proof}
The finiteness of $\cG_\infty^{(N)}(u)$ implies $\mathfrak G_N(R)\to0$ as $R\to\infty$. Proposition~\ref{prop:annular-gradient-error-PB} applies.
\end{proof}

\begin{rem}
The converse of Proposition~\ref{prop:finite-gradient-error} fails: Proposition~\ref{prop:infinite-gradient-error-example} gives $\cG_\infty^{(N)}(u)=\infty$ but $\mathfrak G_N(R)\to0$. Moreover, Remark~\ref{rem:upper-growth-planar} gives the Pohozaev condition at infinity with $\liminf_{R\to\infty}\mathfrak G_2(R)>0$. Thus the criterion in Proposition~\ref{prop:annular-gradient-error-PB} is not necessary either.
\end{rem}

Theorem~\ref{thm:annular-pohozaev} gives the boundary assumption of Proposition~\ref{prop:global-pohozaev}. Under the stated one-sided derivative integrability, the annular estimates also give the following directional identity.

\begin{prop}
Assume the hypotheses of Theorem~\ref{thm:annular-pohozaev}. For a fixed vector $e\in\RN$, set $D_e=(e\cdot\nabla V)e^{Nu}$. If at least one of $(D_e)_+$ and $(D_e)_-$ is integrable, then $D_e\in L^1(\RN)$ and
\[
 \int_{\RN}(e\cdot\nabla V)e^{Nu}\,\d x=0.
\]
In particular, if $|\nabla V|e^{Nu}\in L^1(\RN)$, then
\[
 \int_{\RN}\nabla V\,e^{Nu}\,\d x=0.
\]
\end{prop}

\begin{proof}
Define
\[
 \mathcal T_e=(e\cdot\nabla u)a_N(\nabla u)-e\Phi_N(\nabla u).
\]
Subtracting the translated versions of Lemma~\ref{lem:local-pohozaev} gives
\be \label{eq:local-translation}
 \diver\Big(\mathcal T_e+\frac eN Ve^{Nu}\Big)
 =\frac1N(e\cdot\nabla V)e^{Nu}
 \quad\text{ in }\, \mathcal D'(\RN).
\ee
By Theorem~\ref{thm:annular-pohozaev} and the coarea selection in Proposition~\ref{prop:annular-gradient-error-PB}, there are strictly increasing $r_j\to\infty$ such that, with $q=\nabla(u+\gamma_0\ln  F^0)$,
\[
 r_j\int_{\pa \W_{r_j}}(|Ve^{Nu}|+|q|^N)\,\d\mathcal H^{N-1}\to0.
\]
The estimate $|\mathcal T_e|\le C_F|e|\,|\nabla u|^N$ implies
\[
 \int_{\pa \W_{r_j}}|\mathcal T_e|\,\d\mathcal H^{N-1}
 \le C|e|\Big(r_j^{-1}+
 \int_{\pa \W_{r_j}}|q|^N\,\d\mathcal H^{N-1}\Big)\to0.
\]
The boundary term involving $Ve^{Nu}$ also tends to zero. Integration of \eqref{eq:local-translation} therefore gives $\int_{\W_{r_j}}D_e\,\d x\to0$. The same positive-part/negative-part argument as in Proposition~\ref{prop:global-pohozaev} makes $D_e$ absolutely integrable and identifies its integral as zero. Applying the scalar identity to each coordinate direction proves the vector identity.
\end{proof}

Consequently, if a nonzero direction $e$ satisfies $e\cdot\nabla V\ge0$ everywhere and $e\cdot\nabla V\not\equiv0$, no solution under Theorem~\ref{thm:annular-pohozaev} exists. No independent derivative integrability assumption is needed in this one-sign case. If both the global radial identity and the vector translation identity hold, then
\[
 \int_{\RN}(x-z)\cdot\nabla V\,e^{Nu}\,\d x
 =\int_{\RN}x\cdot\nabla V\,e^{Nu}\,\d x,
 \quad\forall\,z\in\RN.
\]
Thus the integral is independent of the choice of origin under these integrability hypotheses.

\subsection{Proof of Theorem~\ref{thm:rigidity}}

Combining the global identity with the sharp mass threshold produces an explicit lower bound for the weighted radial derivative. When the cone exponent is nonnegative, the relevant expression is monotone in $\gamma_0$.

\begin{prop}
Under the hypotheses of Proposition~\ref{prop:global-pohozaev}, assume additionally that $\gamma\ge0$ and set  $b_\gamma=1+\gamma/N$. Then
\be \label{eq:radial-derivative-bound}
 \int_{\RN}x\cdot\nabla V\,e^{Nu}\, \d x
 \ge N\kap _N b_\gamma^{N-1}((N-1)b_\gamma-N).
\ee
Equality holds if and only if $\gamma_0=b_\gamma$.
\end{prop}

\begin{proof}
For $s\ge1$, the function $H(s)=s^{N-1}((N-1)s-N)$ is strictly increasing, since
\[
 H'(s)=N(N-1)s^{N-2}(s-1).
\]
Apply \eqref{eq:global-pohozaev} and $\gamma_0\ge b_\gamma\ge1$.
\end{proof}

This gives an integral nonexistence condition without pointwise monotonicity of $V$. In particular, a solution is impossible if $\gamma\ge N/(N-1)$ and the integral in \eqref{eq:radial-derivative-bound} is negative, or if $\gamma>N/(N-1)$ and that integral is nonpositive. At $\gamma=N/(N-1)$, a zero integral is consistent with the two identities only when $\gamma_0=N/(N-1)$; these identities do not exclude that equality case.

\begin{proof}[Proof of Theorem~\ref{thm:rigidity}]
Under $x\cdot\nabla V\le0$, Proposition~\ref{prop:global-pohozaev} makes the weighted derivative integrable. Its integral is strictly negative. Indeed, if the integral vanished, continuity and $e^{Nu}>0$ would imply $x\cdot\nabla V\equiv0$. For $r>0$ and $\theta\in\mathbb S^{N-1}$, $\d V(r\theta)/\d r=0$; continuity at zero gives $V(r\theta)=V(0)$. This contradicts \eqref{eq:cone-growth} with $\gamma\ge N/(N-1)>0$. Thus \eqref{eq:global-pohozaev} gives $0<\gamma_0<N/(N-1)$, contrary to Theorem~\ref{thm:threshold}.
\end{proof}

\begin{rem}
The lower bound \eqref{eq:threshold} and the restriction from \eqref{eq:global-pohozaev} have the same threshold when $1+\gamma/N=N/(N-1)$. Proposition~\ref{prop:monotone-sharpness} shows that the resulting exponent is optimal: every smaller nonnegative cone exponent occurs for a sign-changing prescribed function with $x\cdot\nabla V<0$ for $x\ne0$, admitting a solution. For \eqref{eq:larger-upper-growth}, Remark~\ref{rem:upper-growth-planar} proves optimality of the power $2$ for planar annular convergence, not for nonexistence; the higher-dimensional case is not settled here.
\end{rem}

\begin{proof}[Proof of Corollary~\ref{cor:unconditional-rigidity}]
Theorem~\ref{thm:annular-pohozaev} gives the Pohozaev condition at infinity. Theorem~\ref{thm:rigidity} proves nonexistence.
\end{proof}

Combining Theorem~\ref{thm:rigidity} with Proposition~\ref{prop:larger-upper-growth} gives the same nonexistence conclusion when the polynomial upper bound is replaced by \eqref{eq:larger-upper-growth}.

\begin{proof}[Proof of Corollary~\ref{cor:finite-gradient-error}]
Proposition~\ref{prop:finite-gradient-error} gives the Pohozaev condition at infinity, and Theorem~\ref{thm:rigidity} proves nonexistence.
\end{proof}

Corollary~\ref{cor:polynomial} includes $V=c-S_{2m}$, where $c\in\R$ and $S_{2m}\not\equiv0$ is a nonnegative homogeneous polynomial of degree $2m\ge2$. The polynomial may vanish on a subspace or an algebraic cone; only one open cone where $S_{2m}>0$ is used. We verify the growth and monotonicity assumptions below.

\begin{proof}[Proof of Corollary~\ref{cor:polynomial}]
Write $P=P_d+\cdots+P_0$, where $P_j$ is homogeneous of degree $j$ and $P_d\not\equiv0$. For $x=r\theta$, divide $x\cdot\nabla P(x)\le0$ by $r^d$ and let $r\to\infty$. Euler's identity gives $dP_d(\theta)\le0$. Since $P$ is nonconstant, $d\ge1$. If $d$ were odd, $P_d(-\theta)=-P_d(\theta)$ and $P_d\le0$ would give $P_d\equiv0$. Thus $d$ is even and $d\ge2$.

There exists a relatively open spherical patch on which $P_d<0$. On a smaller patch with compact closure, the negativity is uniform, and the lower-order terms are $O(r^{d-1})$ uniformly. Hence $|P(r\theta)|\ge cr^d$ there for large $r$, and $|P(x)|\le C(1+|x|)^d$ on $\RN$. Therefore $P$ satisfies the cone lower bound and the polynomial upper bound in Corollary~\ref{cor:unconditional-rigidity}. Since $d\ge2\ge N/(N-1)$, Corollary~\ref{cor:unconditional-rigidity} applies.

For $V=1+P_m$, Euler's identity gives $x\cdot\nabla V=mP_m\le0$, and the same argument verifies the cone and upper growth bounds.
\end{proof}

\section{Sharpness and geometric consequences}\label{sec:applications}

We construct equality examples and examine the strength of the annular conclusions. In the Euclidean plane the signed mass also gives an obstruction to completeness. We then construct prescribed functions satisfying $x\cdot\nabla V<0$ for $x\ne0$ to prove optimality of the rigidity exponent.

\subsection{Conformal curvature and sharpness}\label{subsec:curvature-sharpness}

For $N=2$ and $F(\xi)=|\xi|$, the metric $g=e^{2u}|\d x|^2$ has Gaussian curvature
\[
 K_g=-e^{-2u}\Delta u=V,\quad
 \int_{\R^2}|K_g|\,\d A_g=\int_{\R^2}|V|e^{2u}\,\d x<\infty.
\]
Theorem~\ref{thm:threshold} therefore gives
\be \label{eq:total-gaussian-curvature-bound}
 \int_{\R^2}K_g\,\d A_g\ge2\pi+\pi\gamma.
\ee
The planar equation gives $u\in C^2_{loc}(\R^2)$. Combining
\eqref{eq:total-gaussian-curvature-bound} with the Cohn--Vossen inequality gives the following consequence; no finite-area assumption is imposed. For related finite-area results, see Lytchak~\cite{Lytchak2023}.

\begin{cor}\label{cor:completeness}
Let $N=2$, $F(\xi)=|\xi|$, and let $u$ satisfy the hypotheses of Theorem~\ref{thm:threshold}. If $g=e^{2u}|\d x|^2$ is complete, then
\[
 2\pi+\pi\gamma\le\int_{\R^2}K_g\,\d A_g\le2\pi,
 \quad  \gamma\le0.
\]
Thus $\gamma>0$ excludes completeness. When $\gamma=0$, completeness implies $M=2\pi$.
\end{cor}

\begin{proof}
The metric has the topology of the plane and finite absolute total curvature. The Cohn--Vossen inequality for complete surfaces therefore gives $\int_{\R^2}K_g\,\d A_g\le2\pi$; see \cite[Satz~6]{CohnVossen1935} and \cite{Huber1957}. Together with \eqref{eq:total-gaussian-curvature-bound}, this proves Corollary~\ref{cor:completeness}.
\end{proof}

For $F(\xi)^2=\xi^TA\xi$ in dimension two, the change of variables $x=A^{1/2}y$ reduces both the equation and the metric in Section~\ref{sec:introduction} to the Euclidean case. Moreover,
\[
 \kap _2=\pi\sqrt{\det A},\quad
 \int_{\R^2}K_g\,\d A_g
 =\frac{M}{\sqrt{\det A}}=\frac{\pi M}{\kap _2}.
\]
Thus completeness again implies $\gamma\le0$; if also $\gamma=0$, then $M=2\kap _2$.

In the Euclidean plane, the signed mass also describes the boundary term in the Gauss--Bonnet formula.
If $k_g$ denotes geodesic curvature with the boundary orientation of $B_R$,
then
\be \label{eq:gauss-bonnet-radius}
 \int_{\pa  B_R}k_g\,\d s_g
 =2\pi-\int_{B_R}Ve^{2u}\,\d x
 \to  2\pi-M\quad\text{ as }\, R\to\infty.
\ee
Thus \eqref{eq:total-gaussian-curvature-bound} makes the limit in
\eqref{eq:gauss-bonnet-radius} at most $-\pi\gamma$.
Corollary~\ref{cor:completeness} compares this curvature deficit with the
Cohn--Vossen inequality; no pointwise metric expansion is used.

Equality in the lower bound can occur when $V$ changes sign. We now construct such a case explicitly; the same construction also shows why a bounded correction to the leading logarithm would be too restrictive.

\begin{prop}\label{prop:equality-example}
Let $\gamma>-N$. Then there exist $V\in C^1(\RN)$ and a bounded-above entire solution $u\in C^2(\RN)$ of \eqref{maineq} such that
\be \label{eq:equality-coefficient-growth}
 c|x|^\gamma\le |V(x)|\le C|x|^\gamma,
 \quad \forall\, x\in\RN\ \text{with }|x|\ge R_0,
\ee
\be \label{eq:equality-mass}
 \int_{\RN}|V|e^{Nu}\, \d x<\infty,
 \quad
 \int_{\RN}Ve^{Nu}\, \d x
 =N\kap _N\Big(1+\frac\gamma N\Big)^{N-1}.
\ee
$V$ is negative outside a compact set and positive somewhere.  The solution has finite exterior gradient error $\cG_\infty^{(N)}(u)<\infty$ and can be chosen so that, with $r=F^0(x)$,
\be \label{eq:equality-profile}
 u(x)=-\Big(1+\frac\gamma N\Big)\ln  r
             -\frac2N\ln \ln  r,
 \quad \forall\, r\ge e^2.
\ee
These examples satisfy \eqref{eq:polynomial-upper} with $\si =\max\{\gamma,0\}$.
\end{prop}

\begin{proof}
Set $\beta=1+\gamma/N>0$ and $d=2/N$.  Choose $\chi\in C^\infty(\R)$ with $0\le\chi\le1$, $\chi=0$ on $(-\infty,1]$, and $\chi=1$ on $[2,\infty)$.  Define
\[
 s(t):=
\left\{\begin{aligned}
  &0,&&t\le1,\\
 &\chi(t)(\beta+d/t),&&t>1,
\end{aligned}\right.
 \quad
 U(t):=-2\beta-d\ln 2-\int_2^t s(\tau)\, \d\tau.
\]
The function $s$ is smooth and nonnegative, and $U$ is constant for $t\le1$.  Set $u(x)=U(\ln  F^0(x))$ for $x\ne0$, extending it by this constant to the origin.  Then $u$ is constant on $\W_e$, belongs to $C^2(\RN)$, and is bounded above.  For $t\ge2$,
\[
 U(t)=-\beta t-d\ln  t,
\]
which is \eqref{eq:equality-profile}.

The dual identities give, for $r=F^0(x)>0$,
\[
 a_N(\nabla u)
 =-s(\ln  r)^{N-1}\frac{x}{r^N}.
\]
This vector field vanishes on a neighborhood of the origin.  Thus the function
\be \label{eq:equality-rhs-formula}
 f(x):=-\QN u(x)
 =r^{-N}\frac{\d}{\d t}(s(t)^{N-1})
       \Big|_{t=\ln  r}
\ee
is well defined and $C^1$ on all of $\RN$, with value zero in $\W_e$.  Set $V=fe^{-Nu}$.  Equation~\eqref{maineq} holds by construction and $V\in C^1(\RN)$.

For $t=\ln  r\ge2$, formula \eqref{eq:equality-rhs-formula} becomes
\[
 f(x)=-(N-1)d(\beta+d/t)^{N-2}r^{-N}t^{-2}.
\]
Since $Nd=2$ and $N\beta-N=\gamma$, we obtain the exterior formula
\[
 V(x)=-(N-1)d\Big(\beta+\frac d{\ln  r}\Big)^{N-2}r^\gamma,
 \quad \forall\, r\ge e^2.
\]
The factor in parentheses is bounded above and below by positive constants.  Norm equivalence therefore proves \eqref{eq:equality-coefficient-growth}.  In particular, $V$ is negative outside $\W_{e^2}$.

Polar integration and the substitution $t=\ln  r$ give
\[
 \int_{\RN}|f|\, \d x
 =N\kap _N\int_{-\infty}^{\infty}
       \Big|\frac{\d}{ \d t}s(t)^{N-1} \Big|\, \d t<\infty.
\]
The derivative is smooth on the finite transition interval, vanishes for $t\le1$, and is $O(t^{-2})$ as $t\to\infty$.  Moreover,
\[
 \int_{\RN}f\, \d x
 =N\kap _N[s(t)^{N-1}]_{t=-\infty}^{t=\infty}
 =N\kap _N\beta^{N-1}.
\]
This proves \eqref{eq:equality-mass}.  The identities $\int_{\RN}f\,\d x>0$ and $f<0$ outside $\W_{e^2}$ imply that $f$, and hence $V$, is positive somewhere.  Moreover, \eqref{eq:equality-profile} gives
\[
 \int_{\{r>e^2\}}F(\nabla u+\beta\nabla\ln  r)^N\, \d x
 =N\kap _Nd^N\int_{e^2}^{\infty}\frac{\d r}{r(\ln  r)^N}<\infty.
\]
The gradient error on $\{1<r<e^2\}$ is finite by regularity.  Finally, for $\gamma\ge0$, continuity on a compact set and \eqref{eq:equality-coefficient-growth} give the global bound $|V(x)|\le C(1+|x|)^\gamma$. For $-N<\gamma<0$, $V$ is bounded near infinity, so continuity gives \eqref{eq:polynomial-upper} with $\si =0$.
\end{proof}

\begin{rem}\label{rem:critical-cone-exponent}
The restriction $\gamma>-N$ is sharp for positivity of the signed mass. In Proposition~\ref{prop:equality-example}, take
\[
 s(t)=0\quad(t\le1),\quad s(t)=\frac{\chi(t)}t\quad(t>1),
 \quad U(2)=-\ln 2,\quad U'=-s.
\]
Then $u=U(\ln  F^0)$ is constant near zero and equals $-\ln \ln  r$ for $r=F^0(x)\ge e^2$. Formula~\eqref{eq:equality-rhs-formula} gives
\[
 f=-(N-1)r^{-N}(\ln  r)^{-N},\quad V=-(N-1)r^{-N}
 \quad(r\ge e^2).
\]
The absolute mass is finite since $N>1$, whereas
$\int_{\RN}f\,\d x=N\kap _N[s^{N-1}]_{-\infty}^{\infty}=0$.
Thus a $C^1$ sign-changing prescribed function can satisfy the cone bound with $\gamma=-N$, and hence also with any $\gamma<-N$, without positive signed mass.
\end{rem}

\begin{prop}\label{prop:infinite-gradient-error-example}
Let $\gamma>-N$ and $1-1/N\le a<1$. There exist a sign-changing $V\in C^1(\RN)$ and a solution $u\in C^2(\RN)$ satisfying~\eqref{eq:basic-class} such that
\[
 \int_{\RN}Ve^{Nu}\,\d x
 =N\kap _N(1+\gamma/N)^{N-1},\quad
 \cG_\infty^{(N)}(u)=\infty.
\]
$V$ satisfies the lower cone bound with exponent $\gamma$ on every cone, and for every $\var >0$ it satisfies a polynomial upper bound with exponent $\max\{\gamma+\var ,0\}$. Outside $\W_{e^2}$ the solution can be chosen as
\[
 u(x)=-\beta\ln  r-(\ln  r)^a,
 \quad  r=F^0(x),\quad  \beta=1+\gamma/N.
\]
Nevertheless, $\mathfrak G_N(R)\to0$ as $R\to\infty$.
\end{prop}

\begin{proof}
Use the cutoff from Proposition~\ref{prop:equality-example} and replace the exterior slope $\beta+d/t$ by $\beta+a t^{a-1}$. Thus, for $t>1$, set $s(t)=\chi(t)(\beta+a t^{a-1})$, set $s=0$ for $t\le1$, and choose $U(2)=-2\beta-2^a$, $U'=-s$. The resulting $u=U(\ln  F^0)$ is constant near zero and belongs to $C^2(\RN)$. Formula~\eqref{eq:equality-rhs-formula}, with $V=fe^{-Nu}$, gives for $t=\ln r\ge2$
\begin{align*}
 f(x)&=-(N-1)a(1-a)(\beta+a t^{a-1})^{N-2}r^{-N}t^{a-2},\\
 V(x)&=-(N-1)a(1-a)(\beta+a t^{a-1})^{N-2}
           r^\gamma t^{a-2}e^{Nt^a}.
\end{align*}
Since $a<1$, $f$ is absolutely integrable. Since $s(-\infty)=0$ and $s(\infty)=\beta$, its signed mass is $N\kap _N\beta^{N-1}$, so $f$ is positive somewhere in the interior. For every $\var>0$,
\[
 t^{a-2}e^{Nt^a}\to\infty,\quad
 t^{a-2}e^{Nt^a-\var t}\to0
 \quad\text{ as }\,t\to\infty.
\]
These limits give the stated lower and upper growth bounds.

On $\{F^0>e^2\}$, the dual identities give
\[
 \int_{\{F^0>e^2\}}F(\nabla u+\beta\nabla\ln F^0)^N\,\d x
 =N\kap _N a^N\int_2^\infty t^{N(a-1)}\,\d t=\infty,
\]
as $N(a-1)\ge-1$. On a fixed-width logarithmic interval, however,
\[
 \mathfrak G_N(R)
 =N\kap _N a^N\int_{\ln R}^{\ln R+\ln2}t^{N(a-1)}\,\d t
 \to0\quad\text{ as }\,R\to\infty,
\]
since $a<1$. This proves the last conclusion.
\end{proof}

These constructions choose $V$ and $u$ simultaneously; they do not prove existence for an arbitrary prescribed $V$ or classify every equality case.

The next example satisfies the hypotheses of Theorem~\ref{thm:annular-pohozaev} with $\limsup_{R\to\infty}\mathfrak G_2(R)>0$.

\begin{prop}\label{prop:annular-subsequence-essential}
In dimension two with $F(\xi)=|\xi|$, there exist $u,V\in C^\infty(\R^2)$ satisfying the hypotheses of Theorem~\ref{thm:annular-pohozaev} with $\gamma=0$ and a polynomial upper bound with exponent $5$, such that
\[
 \int_{\R^2}Ve^{2u}\,\d x=2\pi,\quad
 \liminf_{R\to\infty}\mathfrak G_2(R)=0,
 \quad \limsup_{R\to\infty}\mathfrak G_2(R)\ge2\pi.
\]
Thus neither equality in the mass bound nor the polynomial upper bound implies $\mathfrak G_2(R)\to0$ for all $R\to\infty$.
\end{prop}

\begin{proof}
Take the smooth Euclidean example $u_b$ in Proposition~\ref{prop:equality-example} with $N=2$ and $\gamma=0$. For $r=|x|\ge e^2$,
\[
 u_b(x)=-\ln  r-\ln \ln  r,\quad
 f_b:=-\Delta u_b=-r^{-2}(\ln  r)^{-2}.
\]
Choose a smooth nondecreasing function $a:\R\to[0,1]$ with $a=0$ on $(-\infty,0]$ and $a=1$ on $[1,\infty)$. For $L>2$, set
\[
 b_L(t)=a(t)a(L-t),\quad p_L(t)=\int_0^t b_L(s)\,\d s.
\]
Then $b_L$ vanishes outside $[0,L]$, equals one on $[1,L-1]$, and $\int_\R|b_L'(t)|\,\d t=2$. Set
\[
 L_j=4^{j+2},\quad R_j=e^{L_j},\quad
 x_j=(3R_j/2,0),\quad m_j=L_j^{-1/2}\quad(j\ge1),
\]
and, for $x\ne x_j$, define
\[
 w_j(x)=-m_jp_{L_j}(-\ln |x-x_j|).
\]
The function is zero outside $B_1(x_j)$ and constant when $|x-x_j|\le e^{-L_j}$; extend it by that constant at the center. Flatness of $a$ at the two transition points gives $w_j\in C_c^\infty(\R^2)$. Direct radial integration yields
\begin{equation}\label{eq:localized-perturbation-estimates}
\begin{aligned}
 \int_{\R^2}|\Delta w_j|\,\d x&=4\pi m_j,\quad
 \int_{\R^2}\Delta w_j\,\d x=0,\\
 \int_{\R^2}|\nabla w_j|^2\,\d x
 &=2\pi m_j^2\int_0^{L_j}b_{L_j}(t)^2\,\d t
 \to2\pi\quad\text{ as }\,j\to\infty.
\end{aligned}
\end{equation}
The unit balls are pairwise disjoint and $|x_j|\to\infty$ as $j\to\infty$. Define
\[
 u=u_b+\sum_{j=1}^{\infty}w_j,\quad f=-\Delta u,
 \quad V=fe^{-2u}.
\]
The sum is locally finite, $u\le u_b$, and $u,V$ are smooth. Since $\sum_{j=1}^{\infty}m_j<\infty$, we have $f\in L^1(\R^2)$ and $\int_{\R^2}f\,\d x=\int_{\R^2}f_b\,\d x=2\pi$. The positive signed mass and the negative values of $V$ on an exterior cone imply a sign change. On the cone $x_1<-|x|/2$, the perturbations vanish and $V=-1$ for large $|x|$, proving \eqref{eq:cone-growth}.

We verify \eqref{eq:polynomial-upper}. On $B_1(x_j)$, we have $|x|\asymp R_j$, $|w_j|\le\sqrt{L_j}$, and
\[
 |\Delta w_j|\le C m_j e^{2L_j}=C m_jR_j^2,
 \quad e^{-2u_b}\le C R_j^2L_j^2.
\]
Consequently,
\[
 |V(x)|\le C(1+R_j^4L_j^{3/2})e^{2\sqrt{L_j}}
 \le C R_j^5.
\]
Outside these balls, $V$ equals the smooth function $f_b e^{-2u_b}$, so $|V(x)|\le C(1+|x|)^5$ globally.

Finally, the annulus $\{R_j<|x|<2R_j\}$ contains exactly the $j$th perturbation. Since
\[
 \|\nabla(u_b+\ln |x|)\|_{L^2(\{R_j<|x|<2R_j\})}\to0,
\]
the triangle inequality for the $L^2$ norm gives $\mathfrak G_2(R_j)\to2\pi$. The annulus $\{2R_j<|x|<4R_j\}$ contains no perturbation, so $\mathfrak G_2(2R_j)\to0$. The signed mass gives $\gamma_0=1$, as used in these error integrals.
\end{proof}

\begin{rem}\label{rem:upper-growth-planar}
In the Euclidean plane, the power $2$ in~\eqref{eq:larger-upper-growth}
cannot be increased without losing the annular conclusion. Fix $\vartheta>2$.
There exist smooth $u,V$, with $V$ sign-changing, satisfying~\eqref{eq:basic-class} and
\eqref{eq:cone-growth} with $\gamma=0$, such that
\begin{equation}\label{eq:upper-growth-counterexample}
 \ln(e+|V(x)|)\le C[\ln(e+|x|)]^\vartheta,\quad
 M=2\pi,\quad
 \liminf_{R\to\infty}\mathfrak G_2(R)\ge2\pi.
\end{equation}
Use $u_b$, $b_L$, $p_L$, and $w_j$ from the proof of
Proposition~\ref{prop:annular-subsequence-essential}, replacing the
parameters by
\[
 L_j=j^\vartheta,\quad m_j=j^{-\vartheta/2},\quad
 x_j=((3/2)^j,0),\quad j\ge j_0.
\]
Choose $j_0$ so that the balls $B_1(x_j)$ are pairwise disjoint and
contained in $\R^2\setminus\overline{B_{e^2}}$. Set
$u=u_b+\sum_{j=j_0}^\infty w_j$ and $V=(-\Delta u)e^{-2u}$.
Formula~\eqref{eq:localized-perturbation-estimates}, with
$m_j^2L_j=1$ and $\sum_{j=j_0}^\infty m_j<\infty$, gives finite absolute
mass and $M=2\pi$. The cone $x_1<-|x|/2$ is unchanged, and $u\le u_b$.
On $B_1(x_j)$,
\[
 |w_j|\le j^{\vartheta/2},\quad
 |\Delta w_j|\le C m_j e^{2j^\vartheta},\quad
 e^{-2u_b}\le C(3/2)^{2j}j^2.
\]
Hence $\ln(e+|V|)\le Cj^\vartheta$ there, which proves the growth bound
in~\eqref{eq:upper-growth-counterexample}. Outside the balls, $V$ equals
the bounded prescribed function of $u_b$.

For every sufficiently large $R$, choose $j$ with
$5R/4\le |x_j|<15R/8$. Then $B_1(x_j)\subset B_{2R}\setminus B_R$.
The disjoint supports and~\eqref{eq:localized-perturbation-estimates} give
\[
 \mathfrak G_2(R)^{1/2}
 \ge \|\nabla w_j\|_{L^2(\R^2)}
 -\|\nabla(u_b+\ln|x|)\|_{L^2(B_{2R}\setminus B_R)}
 =\sqrt{2\pi}-o(1)
 \quad\text{ as }\,R\to\infty.
\]
Nevertheless, the Pohozaev condition at infinity holds: the circles of
radii $r_j=\tfrac54(3/2)^j$ avoid all the supports for large $j$, so
$u=u_b$ in their neighborhoods and both limits in~\eqref{eq:PB-boundary}
follow from~\eqref{eq:equality-profile}.
Thus~\eqref{eq:upper-growth-counterexample} proves sharpness for annular
convergence in the Euclidean plane, not necessity of the upper growth
assumption for the Pohozaev condition or Theorem~\ref{thm:rigidity}.
\end{rem}

The same constructions show that the completeness obstruction is sharp and that the total mass does not decide completeness.

\begin{cor}
In the Euclidean plane, for every $-2<\gamma\le0$ the example in Proposition~\ref{prop:equality-example} gives a complete conformal metric with sign-changing Gaussian curvature and total curvature $2\pi+\pi\gamma$. For $\gamma=0$ this metric has finite area and Gaussian curvature $-1$ outside a compact set. There also exists an incomplete finite-area conformal metric satisfying the lower cone bound with $\gamma=0$, a polynomial upper bound, and the same total curvature $2\pi$.
\end{cor}

\begin{proof}
For the first example, write $\beta=1+\gamma/2\in(0,1]$. Outside a compact set its conformal factor is
\[
 e^{u(r)}=r^{-\beta}(\ln  r)^{-1}.
\]
The radial length diverges:
\[
 \int_{e^2}^\infty\frac{r^{-\beta}}{\ln  r}\,\d r=\infty.
\]
For any curve tending to infinity, the metric length is bounded below by the integral of this radial factor against the absolute variation of $r$. The same divergence therefore proves completeness. When $\beta=1$, the exterior area is $2\pi\int_{e^2}^\infty [r(\ln  r)^2]^{-1}\,\d r<\infty$, and the function $V$ in Proposition~\ref{prop:equality-example} is $V=-1$ there.

For the incomplete metric, use Proposition~\ref{prop:infinite-gradient-error-example} with $N=2$, $\gamma=0$, and $a=1/2$. Then $u(r)=-\ln  r-\sqrt{\ln  r}$ for large $r$, and
\[
 \int_{e^2}^\infty e^{u(r)}\,\d r
 =\int_2^\infty e^{-\sqrt t}\,\d t<\infty,\quad
 2\pi\int_{e^2}^\infty r e^{2u(r)}\,\d r
 =2\pi\int_2^\infty e^{-2\sqrt t}\,\d t<\infty.
\]
A radial curve reaching infinity has finite length, so the metric is incomplete. Proposition~\ref{prop:infinite-gradient-error-example} supplies the mass, sign change, and growth bounds.
\end{proof}

For $\gamma=0$, the complete example has, in the coordinates
$t=\ln r$ and $\theta\in\R/(2\pi\mathbb Z)$, the exterior metric
$g=t^{-2}(\d t^2+\d\theta^2)$, a hyperbolic cusp.
Such constant negative-curvature cusp asymptotics are classical;
see Chou--Wan~\cite[Theorem~2]{ChouWan1995}.
The two examples with mass $2\pi$ show that the leading logarithmic term
and total curvature do not determine completeness.
Theorem~\ref{thm:annular-pohozaev} concerns normalized functions and gradients; no convergence of the induced metric distances is used here.

\subsection{Sharpness of the rigidity exponent}

The functions $V$ in Proposition~\ref{prop:equality-example} need not satisfy $x\cdot\nabla V\le0$. We now construct a monotone family to prove the optimality of the exponent in Theorem~\ref{thm:rigidity}.

\begin{prop}\label{prop:monotone-sharpness}
Let $0\le\gamma<N/(N-1)$. There exist a sign-changing $V\in C^1(\RN)$ and a solution $u$ in \eqref{eq:basic-class} such that
\be \label{eq:monotone-sharpness-conditions}
 x\cdot\nabla V(x)<0\quad(x\ne0),\quad
 c|x|^\gamma\le |V(x)|\le C|x|^\gamma\quad(|x|\ge R_0).
\ee
The solution has finite exterior gradient error and signed mass
\be \label{eq:monotone-sharpness-mass}
 M=N\kap _N\beta^{N-1},\quad
 1+\frac\gamma N<\beta<\frac N{N-1}.
\ee
Consequently, the cone exponent in Theorem~\ref{thm:rigidity} and Corollary~\ref{cor:unconditional-rigidity} cannot be decreased.
\end{prop}

\begin{proof}
Set
\be \label{eq:monotone-sharpness-parameters}
 q=\frac N{N-1},\quad
 a=\frac12\Big(1-\frac\gamma q\Big),\quad
 \beta=\frac{N-a}{N-1},\quad d=\frac{a}{2(N-1)}.
\ee
Thus $0<a\le1/2$, $\beta>1$, and $\gamma=q(1-2a)$. With $r=F^0(x)$ and $t=1+r^q$, define
\be \label{eq:monotone-sharpness-profile}
 u(x)=-\frac\beta q\ln  t+\frac1{2N}t^{-a}.
\ee
Since $r^q\in C^1(\RN)$ and $\nabla(r^q)=O(r^{q-1})$ at zero, $u\in C^1(\RN)\cap W^{1,N}_{loc}(\RN)$. The function in \eqref{eq:monotone-sharpness-profile} is strictly decreasing in $r$ and is bounded above by $1/(2N)$.

For $t\ge1$, write
\[
 D(t)=dt^{-a},\quad A(t)=\beta+D(t),\quad
 B(t)=\beta+(1+a-at)D(t).
\]
The identities \eqref{eq:polar-identities} and $q(N-1)=N$ give
\be \label{eq:monotone-sharpness-equation}
 a_N(\nabla u)=-\Big(\frac{A(t)}t\Big)^{N-1}x,
 \quad f:=-Q_Nu=N t^{-N}A(t)^{N-2}B(t).
\ee
The vector field in \eqref{eq:monotone-sharpness-equation} is $C^1$ at the origin, so the identity holds weakly on all of $\RN$, with no singular measure at zero. Define
\be \label{eq:monotone-sharpness-coefficient}
 V(x)=f(x)e^{-Nu(x)}
 =N t^{-a}A(t)^{N-2}B(t)\exp(-\tfrac12t^{-a}).
\ee
Here $N\beta/q-N=-a$. Formula \eqref{eq:monotone-sharpness-coefficient} also gives $V\in C^1(\RN)$.

We verify $x\cdot\nabla V<0$ for $x\ne0$. Set
\[
 P(t)=N t^{-a}A(t)^{N-2}\exp(-\tfrac12t^{-a})>0,
 \quad \mathcal V(t)=P(t)B(t).
\]
Then $V(x)=\mathcal V(1+r^q)$, so it suffices to prove $\mathcal V^{\prime}(t)<0$ for $t\ge1$.
Differentiation gives
\be \label{eq:monotone-sharpness-derivatives}
 \begin{aligned}
 \frac{tP'(t)}{P(t)}
 &=-a+(N-1)D(t)-(N-2)\frac{aD(t)}{A(t)}:=C(t),\\
 tB'(t)&=-aD(t)(1+a+(1-a)t).
 \end{aligned}
\ee
Since $A(t)>\beta>1>a$ and $(N-1)D(t)\le a/2$, \eqref{eq:monotone-sharpness-derivatives} gives
\[
 -a<C(t)\le-a/2<0,\quad B'(t)<0.
\]
If $B(t)\ge0$, then $t\mathcal V^{\prime}(t)/P(t)=B(t)C(t)+tB'(t)<0$. If $B(t)<0$, the lower bound for $C(t)$ yields
\[
 \begin{aligned}
 \frac{t\mathcal V^{\prime}(t)}{P(t)}
 &\le -aB(t)+tB'(t)\\
 &=-a\beta-2a(1+a)D(t)-a(1-2a)tD(t)<0.
 \end{aligned}
\]
Thus $\mathcal V^{\prime}(t)<0$ for every $t\ge1$, and $x\cdot\nabla V=q r^q\mathcal V^{\prime}(t)<0$ for $x\ne0$.

The value $B(1)=\beta+d$ is positive, whereas
$B(t)\sim-ad\,t^{1-a}$ as $t\to\infty$. Hence $V$ changes sign, and
\be \label{eq:monotone-sharpness-asymptotics}
 V(x)\sim-Nad\,\beta^{N-2}r^\gamma,
 \quad f(x)\sim-Nad\,\beta^{N-2}r^{-N-qa}\quad\text{ as }\, r\to\infty.
\ee
Norm equivalence gives the two-sided bound in \eqref{eq:monotone-sharpness-conditions}. The second estimate in \eqref{eq:monotone-sharpness-asymptotics} and local boundedness give $f\in L^1(\RN)$. Integrating \eqref{eq:monotone-sharpness-equation} on $W_R$, and then letting $R\to\infty$, gives
\[
 \int_{W_R}f\,\d x
 =N\kap _N\Big(\frac{R^q A(1+R^q)}{1+R^q}\Big)^{N-1}
  \to N\kap _N\beta^{N-1}\quad\text{ as }\, R\to\infty.
\]
Moreover, \eqref{eq:monotone-sharpness-parameters} gives
$\beta-(1+\gamma/N)=a/(N-1)>0$ and $q-\beta=a/(N-1)>0$, proving \eqref{eq:monotone-sharpness-mass}.

Finally, \eqref{eq:monotone-sharpness-mass} gives $\gamma_0=\beta$, and
\eqref{eq:monotone-sharpness-profile} yields
$F(\nabla u+\gamma_0\nabla\ln r)=O(r^{-1-qa})$ at infinity. Polar integration proves $\cG_\infty^{(N)}(u)<\infty$. Thus the examples satisfy all the remaining hypotheses in the rigidity statements, including the boundary condition. The case $\gamma=0$ also satisfies the lower cone bound for every $-N<\gamma<0$.
\end{proof}

\medskip

		\noindent\textbf{Acknowledgements.}~ 	
Y. Liu is  supported by the National Natural Science
Foundation of China (grants 12201089 and 12371051), the Chongqing
Normal University Foundation (No. BWQB2023006), and the Open Research Fund
of the Key Laboratory of Nonlinear Analysis and Applications, Ministry of Education, Central China Normal University (No. NAA2025ORG006). Z. Tang is supported by NSFC grants 12471099 and 12271223.
H. Wang is supported by China Postdoctoral Science Foundation 2026M793418.
\medskip

		\noindent\textbf{Data availability statement.}~	 Data sharing not applicable to this article as no datasets were generated or analysed during the current study.

\medskip

		\noindent\textbf{Declarations.}~ \textbf{Conflict of interest} ~ On behalf of all authors, the corresponding author states that there is no conflict of interest.

\medskip

\noindent\textbf{AI assistance statement.} ~The initial version of this manuscript was completed by the authors in August 2026. OpenAI models were subsequently used for language editing, consistency checks, and assistance in drafting mathematical revisions. The authors are responsible for independently verifying all revised arguments and references before submission.

\end{document}